\documentclass[12pt]{amsart}
\usepackage{latexsym}
\usepackage{amssymb} 
\usepackage{mathrsfs}
\usepackage{amsmath}
\usepackage{latexsym}
\usepackage{delarray}
\usepackage{amssymb,amsmath,amsfonts,amsthm,
mathrsfs}

\RequirePackage{srcltx}

\usepackage{hyperref}
\renewcommand\eqref[1]{(\ref{#1})} 

\numberwithin{equation}{section}
\theoremstyle{plain}
\newtheorem{thm}{Theorem}[section]
\newtheorem{lem}[thm]{Lemma}
\newtheorem{prop}[thm]{Proposition}
\newtheorem{cor}[thm]{Corollary}

\newtheorem{remark}[thm]{Remark}

\newcommand{\mb}[1]{\ensuremath{\mathbb{#1}}}

\newcommand{\R}{\mb{R}}
\newcommand{\C}{\mb{C}}
\newcommand{\Z}{\mb{Z}}
\newcommand{\T}{\mb{T}}

\def\B{{\mathcal B}}
\def\Hcal{{\mathcal H}}
\def\Dcal{{\mathcal D}}

\newcommand{\dd}{ {\rm d}}
\newcommand{\Gh}{ {\widehat{G}}}
\newcommand{\Lap}{\mathcal{L}}
\newcommand\Rn{{\mathbb R}^n}
\newcommand\Tn{{\mathbb T}^n}

\def\sumxi{\sum_{\xi\in\Gh}}
\def\dxi{{d_\xi}}

\def\HS{{\mathtt{HS}}}
\def\whf{\widehat{f}}
\def\whfhs{\|\widehat{f}(\xi)\|_{\HS}}
\def\sumxip{{\sum_{2^{s}\le\jp{\xi}<2^{s+1}}}}

\newcommand{\wt}[1]{\widetilde{#1}}

\newcommand{\jp}[1]{\langle{#1}\rangle}
\newcommand{\p}[1]{\left({#1}\right)}

\newcommand{\Tr}{{\mathrm{Tr}}}
 \newcommand{\supp}{\operatorname{supp}}

 \newcommand{\conv}{\operatorname{conv}}

\title[Besov, Wiener and Beurling spaces]{
Nikolskii inequality and Besov, Triebel-Lizorkin,
Wiener and Beurling spaces on compact homogeneous manifolds}

\author[Erlan Nursultanov]{Erlan Nursultanov}
\address{Erlan Nurlustanov:
\endgraf
Department of Mathematics
\endgraf
Lomonosovs Moscow State University, Kazakh Branch
\endgraf
and
Gumilyov Eurasian National University,
\endgraf
Astana
\endgraf
Kazakhstan
  \endgraf
  {\it E-mail address} {\rm er-nurs@yandex.ru}
}

\author[Michael Ruzhansky]{Michael Ruzhansky}
\address{
  Michael Ruzhansky:
 \endgraf
  Department of Mathematics
  \endgraf
  Imperial College London
  \endgraf
  180 Queen's Gate, London SW7 2AZ
  \endgraf
  United Kingdom
  \endgraf
  {\it E-mail address} {\rm m.ruzhansky@imperial.ac.uk}
  }

 \author[Sergey Tikhonov]{Sergey Tikhonov}

\address{Sergey Tikhonov:
\endgraf
ICREA
and Centre de Recerca Matem\`{a}tica (CRM)  \endgraf
 E-08193, Bellaterra  \endgraf
 Barcelona
  \endgraf
  {\it E-mail address} {\rm stikhonov@crm.cat}
}

\thanks{E.N. was supported by the Ministry of Education and Science of the Republic of Kazakhstan Grants 1412/GF and 1080/GF.
M.R. was supported by the
EPSRC Leadership Fellowship EP/G007233/1 and by EPSRC Grant EP/K039407/1.
S.T. was
partially supported by the MTM2011-27637/MTM, 2009 SGR 1303, RFFI 12-01-00169, NSH 979.2012.1.
}
\date{}

\subjclass[2010]{Primary 35G10; 35L30; Secondary 46F05;}
\keywords{Nikolskii's inequality, Besov spaces,
Triebel-Lizorkin spaces, Wiener spaces, Beurling spaces,
embeddings, interpolation}

\begin{document}

\begin{abstract}
In this paper we prove Nikolskii's inequality on general compact Lie groups
and on compact homogeneous spaces with the constant interpreted in terms
of the eigenvalue counting function of the Laplacian on the space, giving the
best constant for certain indices, attained on the Dirichlet kernel.
Consequently, we establish embedding theorems between Besov spaces
on compact homogeneous spaces, as well as embeddings between Besov
spaces and Wiener and Beurling spaces. We also analyse Triebel--Lizorkin spaces and
$\beta$-versions of Wiener and Beurling spaces and their embeddings,
and interpolation properties of all these spaces.
\end{abstract}

\maketitle

\section{Introduction}

In this paper we analyse the families of Besov, Triebel--Lizorkin, Wiener and Beurling spaces on compact
Lie groups $G$ and on compact homogeneous manifolds $G/K$. To a large extent, the analysis
is based on establishing an
appropriate version of Nikolskii's inequality in this setting and on working with
discrete Lebesgue spaces on the unitary dual of the group and its
class I representations.

The classical Nikolskii inequality for trigonometric polynomials $T_L$ of degree at
most $L$ on the circle was given by (\cite{nikol})
$$
 \|T_L\|_{L^q(\T)} \le C L^{1/p-1/q} \|T_L\|_{L^p(\T)},
$$
where $1\le p<q\le \infty$ and the constant $C$ can be taken as $2$. A similar result is also known (\cite{nikol}) on the Euclidean space for entire functions of exponential type.
Moreover, for $ f\in L^p(\R^n)$ such that $\supp (\widehat{f})$ is compact we have (see \cite{nessel})
\begin{equation}\label{nik-first}
 \|f\|_{L^q(\R^n)} \le \left(C(p) \mu(\conv [\supp (\widehat{f})]) \right)^{1/p-1/q}  \|f\|_{L^p(\R^n)},
\end{equation}
where $1\le p\le q\le \infty$, $\mu(E)$ denotes the Lebesgue measure of $E$, and
$\conv [E]$  denotes the convex hull of $E$.
Inequalities of type (\ref{nik-first}) are sometimes called the Plancherel-Polya-Nikolskii inequality.

The Nikolskii inequality plays a key role in the investigations of properties of different
function spaces (see, e.g., \cite{nikol1, tribel}), in approximation theory
(see, e.g., \cite{devore}), or to obtain embedding theorems (see, e.g., \cite{tribel, ditzian}).

The Nikolskii inequality for the spherical polynomials on $\mathbb{S}^{n}$ was proved in
\cite[Th. 1]{Kamzolov:sphere-1984}, see also \cite{dai, mhaskar}.
Moreover,  Pesenson in \cite{pes} obtained
a very general version of Bernstein--Nikolskii type  inequalities on non-compact
symmetric spaces, and in \cite{Pesenson:Besov-2008} on compact homogeneous spaces.

In this paper, by a very different method of proof,
we extend Pesenson's result \cite{Pesenson:Besov-2008} to a wider range of indices
$0<p<q\leq\infty$ as well as give an interpretation of the constant that we obtain in terms
of the eigenvalue counting function for the Laplace operator on the group
acting on the homogeneous
manifold. For certain indices this gives the best constant in the Nikolskii
inequality, and this constant is attained on the Dirichlet kernel.

Consequently, we use this to establish embedding properties between different
families of function spaces. Besov spaces on Lie groups have been recently
actively analysed, e.g., from the point of view of the heat kernel
\cite{Skr} and the Littlewood-Paley theory \cite{FMV},
see also \cite{Gallagher-Sire:Besov-algebras-groups-polynomial-growth-SM-2012},
but apart from general
definitions and certain properties no embedding properties for these spaces have been given.
For functions on $\mathbb{S}^{n}$, the Besov spaces were studied in, e.g., \cite{hesse}.
In that paper the Nikolskii inequality was applied to obtain the Sobolev-type embedding
theorems. For more general results see \cite{DW, DX}.

Apart from embeddings between Besov spaces that can be obtained by trivial arguments,
the Nikolskii inequality allows us to derive  a rather complete list of
embeddings with respect to all three indices of the space. The norms we use for
proofs depend on the
number $k_{\xi}$ of invariant vector fields in the representation space
(for a representation $[\xi]\in\widehat{G}$ of dimension $d_{\xi}$) with respect to
the subgroup $K$. In Section \ref{SEC:localisations}, as a consequence
of the interpolation theorems, we show that for certain ranges of indices the
Besov spaces defined in terms of the global Littlewood-Paley theory
agree with the Besov spaces defined in local coordinates.

The setting of this paper provides a general environment when ideas resembling the
classical analysis dealing with the Fourier series can still be carried out.

The overall analysis of this paper relies on working with discrete Lebesgue spaces
$\ell^{p}(\Gh)$ on the unitary dual $\Gh$ of the compact Lie group $G$.
Such spaces have been introduced and developed recently in
\cite{rt:book} and they allow one to quantify the Fourier transforms of functions on
$G$ by fixing the Hilbert-Schmidt norms of the Fourier coefficients and working
with weights expressed in terms of the dimensions of the representations of the
group. These spaces $\ell^{p}(\Gh)$ have been already effective and
were used in \cite{dr:gevrey} to give the characterisation of Gevrey spaces
(of Roumeau and Beurling types)
and the corresponding spaces of ultradistributions on compact Lie groups
and homogeneous spaces.

In this paper we modify this construction and extend it further to enable
one to work with class I representations 
 thus
facilitating the analysis of functions which are
constant on right cosets leading to the quantification of
the Fourier coefficients on the compact homogeneous spaces $G/K$.

Typical examples of homogeneous spaces are
the spheres ${\mathbb S}^{n}={\rm SO}(n+1)/{\rm SO}(n)$
 in which case we can take $k_{\xi}=1$.
Similarly, we can consider  complex spheres
(or complex projective spaces) $\mathbb P\mathbb C^{n}={\rm SU}(n+1)/{\rm SU}(n)$,
or quaternionic projective spaces
$\mathbb P\mathbb H^{n}={\rm Sp}(n+1)/{\rm Sp}(n)\times {\rm Sp}(1)$.
In the case of the trivial subgroup $K=\{e\}$ the
homogenous space is the compact Lie group $G$ itself, and we recover the original spaces in \cite{rt:book} by taking $k_{\xi}=d_{\xi}.$

Consequently, we look at the Wiener algebra $A$ of functions with summable
Fourier transforms and its $\ell^{\beta}$-versions, the $\beta$-Wiener spaces
$A^{\beta}$. In particular, we prove the embedding
$B_{2,1}^{n/2}\hookrightarrow A$ as well as its
$\beta$-version as the embeddings between $A^{\beta}$ and $B_{p,\beta}^{r}$
(and their inverses
depending on whether $1<p\leq 2$ or $2\leq p<\infty$). We note that our version of
these spaces is based on the scale of $\ell^{p}$ spaces described above
(given in \eqref{EQ:norm}) and not on the Schatten norms. Spaces with Schatten
norms have been considered as well as their weighted versions, see
\cite{Ludwig-Spronk-Turowska:Beurling-Fourier-JFA-2012}, and also
\cite{Dales-Lau:Beurling-MAMS-2005} and \cite{Lee-Samei:Beurling-JFA-2012}
for rather extensive analysis. Such spaces go under the name of Beurling--Fourier
spaces in the literature. To distinguish with Beurling spaces described below
(which appear to be almost new in the recent literature especially since
they are based on the different scale of $\ell^{p}$ spaces developed
in \cite{rt:book}), we use the name of
Wiener spaces for the spaces $A^\beta$ referring to the original studies of Wiener
of what is also known as the Wiener algebra.

The Beurling space $A^*$ was introduced by Beurling \cite{Beurling:AM-synthesis}
for establishing contraction
properties of functions.
In \cite{bel} it was shown that $A^*(\T)$ is an algebra and its properties
were investigated.
The definition of the Beurling space in multi-dimensions, even on $\Tn$, is
not straightforward since we would need to take into account the sums in different
directions which can be done in different ways.
For this, we first reformulate the norm of the space $A^{*}$ in the
(Littlewood--Paley) way which
allows extension to spaces when the unitary dual $\Gh$ is discrete but
is different from $\Z^{n}.$
Consequently, we analyse the space $A^*$ as well as its
$\beta$-version $A^{*,\beta}$ in the setting of general compact homogeneous spaces.
These function spaces play an important role in the summability theory and
in the Fourier synthesis (see, e.g.,
\cite[Theorems 1.25 and 1.16]{stein-weiss} and \cite[Theorem 8.1.3]{trigub}). 
On the circle,
the spaces $A^{*,\beta}$ were studied in \cite[Ch. 6]{trigub}.
Here, we analyse an analogue of these spaces on general compact homogeneous groups
and prove two-sided embedding properties between these spaces
and appropriate Besov spaces. The analysis is again based on the scale
of $\ell^{p}$ spaces in \eqref{EQ:norm} which allow us to also
establish their interpolation properties.

We note that the questions of estimating projectors to individual eigenspaces
(rather than to eigenspaces corresponding to eigenvalues $\leq L$ as it arises
in the Nikolskii inequality) have appear naturally
in different problems in harmonic analysis and have been also studied.
For example, in the analysis related to the Carleson-Sj\"olin theorem
for spherical harmonics on the 2-sphere,
Sogge \cite[Theorem 4.1]{Sogge:Oscillatory-intergrals-spherical-harmonics-DMJ-1986}
obtained $L^{2}$-$L^{p}$ estimates for harmonic projections on spheres.
The same $L^{2}$-$L^{p}$ estimates but on compact Lie groups have been obtained
by Giacalone and Ricci
\cite{Giacalone-Ricci:harmonic-projectors-cpt-Lie-gps-MA-1988}.
We note that on the one hand, estimates for the projection to the eigenspace
corresponding to an eigenvalue $L$ are better
than those appearing when projecting to the span of eigenspaces corresponding
to eigenvalues $\leq L$. But on the other hand, the Nikolskii inequality
provides a better estimate compared to the one that can be obtained by
summing up the individual ones. In
Corollary \ref{THM3-2} we show that for certain indices the power in the projection
to to the span of eigenspaces corresponding
to eigenvalues $\leq L$ on a compact Lie group can be improved compared
to the one in the Nikolskii's inequality (but which, in turn, is also sharp for certain indices,
see Theorem \ref{THM2}).

The paper is organised as follows.
Section \ref{SEC:Fourier} is devoted to introducing the spaces $\ell^{p}(\Gh_{0})$
for the class I representations of the compact Lie group $G$ and for the
corresponding Fourier analysis.
Sections \ref{SEC:Nikolskii} and \ref{SEC:Nikolskii1} are devoted to Nikolskii's inequality.
There, in Section \ref{SEC:Nikolskii} we establish the Nikolskii inequality on
general compact homogeneous spaces, and in Section \ref{SEC:Nikolskii1}
we give its refinement of compact Lie groups taking into account the number
of non-zero Fourier coefficients in the constant.
In Section \ref{SEC:Besov} we analyse embedding properties between various smooth function spaces including Sobolev, Besov, and Triebel--Lizorkin spaces. Section \ref{SEC:Wiener}
and Section \ref{SEC:Beurling} deal with Wiener  and Beurling spaces,
respectively. In Section \ref{interp} we establish interpolation properties
of the introduced spaces. In Section \ref{SEC:localisations} we show that
for certain ranges of indices, the Besov spaces introduced in this paper by
the global Littlewood-Paley theory agree with the Besov spaces in local
coordinates.

For an index $1\leq p\leq\infty$ we will always write $p'$ for its dual index
defined by $\frac 1p +\frac{1}{p'}=1$.
For $a,b\geq 0$, we write $a\lesssim b$ is there is a constant $C>0$ such that
$a\lesssim Cb$, and we write $a\asymp b$ if
$a\lesssim b$ and $b\lesssim a$.

\section{Fourier analysis on homogeneous spaces}
\label{SEC:Fourier}

In this section we collect the necessary facts concerning the Fourier analysis on
compact homogeneous spaces. We start with compact Lie groups.

Let $G$ be a compact Lie group of dimension $\dim G$ with its unitary dual
denoted by $\Gh$.
Fixing the basis in representation spaces, we can work with
matrix representations $\xi:G\to\C^{d_{\xi}\times d_{\xi}}$ of degree $d_{\xi}$.
We recall that by the Peter--Weyl theorem the collection
$\{\sqrt{d_{\xi}}\xi_{ij}: [\xi]\in\Gh, 1\leq i,j\leq d_{\xi}\}$ forms an orthonormal
basis in $L^{2}(G)$ with respect to the normalised Haar measure on $G$.
The integrals and the spaces $L^{p}(G)$ are always taken
with respect to the normalised bi-invariant Haar measure on $G$.
For $f\in C^{\infty}(G)$, its Fourier coefficient at $\xi\in [\xi]\in\Gh$ is defined as
$$
\widehat{f}(\xi)=\int_{G} f(x) \xi(x)^{*} \dd  x.
$$
Consequently, we have $\widehat{f}(\xi)\in\C^{d_{\xi}\times d_{\xi}}.$
The Fourier series of $f$ becomes
\begin{equation}\label{EQ:FS}
f(x)=\sum_{[\xi]\in\Gh} d_{\xi} \Tr (\widehat{f}(\xi) \xi(x)).
\end{equation}
In \cite{rt:book}, the Lebesgue spaces $\ell^{p}(\Gh)$  have been introduced
on $\Gh$
by the following norms, which we now write for the Fourier coefficients of $f$, as
\begin{equation}\label{EQ:norm}
\|\widehat{f}\|_{\ell^{p}(\Gh)}=\left(\sum_{[\xi]\in\Gh} d_{\xi}^{p(\frac{2}{p}-\frac12)}
\|\widehat{f}(\xi)\|_{\HS}^{p}\right)^{1/p},\; 1\leq p<\infty,
\end{equation}
and
\begin{equation}\label{EQ:norm-linfty}
\|\widehat{f}\|_{\ell^{\infty}(\Gh)}=
\sup_{[\xi]\in\Gh} d_{\xi}^{-\frac12} \|\widehat{f}(\xi)\|_{\HS},
\end{equation}
with
$\|\widehat{f}(\xi)\|_{\HS}=\Tr(\widehat{f}(\xi) \widehat{f}(\xi)^{*})^{1/2}.$
These are interpolation spaces, and the Hausdorff-Young inequality holds
in them in both directions depending on $p$. We refer to \cite[Section 10.3.3]{rt:book}
for these and other properties of these spaces, and to
\cite{rt:book,rt:groups} for further operator analysis.

For $[\xi]\in\Gh$, we denote by $\jp{\xi}$ the eigenvalue of the operator
$(I-\Lap_{G})^{1/2}$ corresponding to the representation class
$[\xi]\in\Gh$ where $\Lap_{G}$ is the Laplace-Beltrami operator
(Casimir element) on $G$.

We now give modifications of this construction for compact homogeneous
spaces $M$. Let $G$ be a compact motion group of $M$ and $K$ a
stationary subgroup at some point, so that we can identify $M=G/K$.
Typical examples are the spheres ${\mathbb S}^{n}={\rm SO}(n+1)/{\rm SO}(n)$
or complex spheres $\mathbb P\mathbb C^{n}={\rm SU}(n+1)/{\rm SU}(n)$,
or quaternionic projective spaces
$\mathbb P\mathbb H^{n}$.
We normalise measures so
that the measure on $K$ is a probability one.

We denote by $\Gh_{0}$ the subset of $\Gh$ of representations that are
class I with respect to the subgroup  $K$. This means that
$[\xi]\in\Gh_{0}$ if $\xi$ has at least one non-zero invariant vector $a$ with respect
to $K$, i.e., that
$\xi(k)a=a$ for all $k\in K.$
Let $\Hcal_{\xi}\simeq \C^{d_{\xi}}$
denote the representation space of $\xi(x):\Hcal_{\xi}\to\Hcal_{\xi}$
and let $\B_{\xi}$ be the space of these invariant vectors.
Let $k_{\xi}:=\dim\B_{\xi}.$
We fix an orthonormal basis of $\Hcal_{\xi}$ so that
its first $k_{\xi}$ vectors are the basis of $B_{\xi}.$
The matrix elements $\xi_{ij}(x)$, $1\leq j\leq k_{\xi}$,
are invariant under the right shifts by $K$.

We note that if $K=\{e\}$ so that $M=G/K=G$ is the Lie group, we have
$\Gh=\Gh_{0}$ and $k_{\xi}=d_{\xi}$ for all $\xi$. As the other extreme, if
$K$ is a massive subgroup of $G$, i.e., if for every $\xi$ there is precisely one invariant vector
with respect to $K$, we have $k_{\xi}=1$ for all $[\xi]\in\Gh_{0}.$
This is, for example, the case for the spheres $M={\mathbb S}^{n}.$
Other examples can be found in Vilenkin \cite{Vilenkin:BK-eng-1968}.

For a function $f\in C^{\infty}(G/K)$ we can
write the Fourier series of its canonical
lifting $\wt{f}(g):=f(gK)$ to $G$,
$\wt{f}\in C^{\infty}(G)$, so that the Fourier
coefficients satisfy  $\widehat{\wt{f}}(\xi)=0$ for all representations
with $[\xi]\not\in\Gh_{0}$.
Moreover, for class I representations we have $\widehat{\wt{f}}(\xi)_{{jk}}=0$
for $j>k_{\xi}$. We will often drop writing tilde for simplicity, and agree that
for a distribution $f\in\Dcal'(G/K)$ we have
$\widehat{f}(\xi)=0$ for $[\xi]\not\in\Gh_{0}$ and
$\widehat{f}(\xi)_{ij}=0$ if $i>k_{\xi}$.

With this, we can write the Fourier series of $f$ (or of $\wt{f}$) in terms of
the spherical functions $\xi_{ij}$, $1\leq j\leq k_{\xi}$, of the representations
$\xi$, $[\xi]\in\Gh_{0}$, with respect to the subgroup
$K$. Namely, the Fourier series \eqref{EQ:FS} becomes
\begin{equation}\label{EQ:FSh}
f(x)=\sum_{[\xi]\in\Gh_{0}} d_{\xi} \sum_{i=1}^{d_{\xi}} \sum_{j=1}^{k_{\xi}}
\widehat{f}(\xi)_{ji}\xi_{ij}(x).
\end{equation}
For the details of this construction we refer to \cite{Klimyk-Vilenkin:Vol-1}.

In the case we work on the homogeneous space $G/K$, in order to shorten the
notation, for $[\xi]\in\Gh_{0}$,
it makes sense to set $\xi(x)_{ij}:=0$ for all $j>k_{\xi}$. Indeed, this will
not change the Fourier series expression \eqref{EQ:FSh} since these entires
do not appear in the sum \eqref{EQ:FSh} for $f\in\Dcal'(G/K).$ With this convention
we can still write \eqref{EQ:FSh} in the compact form
\begin{equation}\label{EQ:FSh-mat}
f(x)=\sum_{[\xi]\in\Gh_{0}} d_{\xi} \Tr (\widehat{f}(\xi) \xi(x)),\quad
f\in C^{\infty}(G/K).
\end{equation}
For future use, we note that with these conventions the matrix
$\xi(x)\xi(x)^{*}$ is diagonal with the first $k_{\xi}$ diagonal entries equal to one and
others equal to zero, so that we have
\begin{equation}\label{EQ:xi-HS}
\|\xi(x)\|_{\HS}=k_{\xi}^{1/2} \textrm{ for all } [\xi]\in\Gh_{0}, \; x\in G/K.
\end{equation}
For the space of Fourier coefficients of functions on $G/K$ we define
\begin{equation}\label{EQ:Sigma}
\Sigma(G/K):=\{\sigma:\xi\mapsto\sigma(\xi): \;
[\xi]\in\Gh_{0},\; \sigma(\xi)\in\C^{d_{\xi}\times d_{\xi}}, \; \sigma(\xi)_{ij}=0
\textrm{ for } i>k_{\xi}\}.
\end{equation}
In analogy to \eqref{EQ:norm},
we can define the Lebesgue spaces $\ell^{p}(\Gh_{0})$
by the following norms which we will apply to Fourier coefficients
$\widehat{f}\in\Sigma(G/K)$ of $f\in\Dcal'(G/K)$.
Thus, for $\sigma\in\Sigma(G/K)$ we set
\begin{equation}\label{EQ:Lp-sigmaGK}
\|\sigma\|_{\ell^{p}(\Gh_{0})}:=\left(\sum_{[\xi]\in\Gh_{0}} d_{\xi}
k_{\xi}^{p(\frac{1}{p}-\frac12)}
\|\sigma(\xi)\|_{\HS}^{p}\right)^{1/p},\; 1\leq p<\infty,
\end{equation}
and
$$
\|\sigma\|_{\ell^{\infty}(\Gh_{0})}:=
\sup_{[\xi]\in\Gh_{0}} k_{\xi}^{-\frac12} \|\sigma(\xi)\|_{\HS}.
$$
In the case $K=\{e\}$, so that $G/K=G$, these spaces coincide with
those defined by \eqref{EQ:norm} since $k_{\xi}=d_{\xi}$ in this case.
Again, by the same argument as that in \cite{rt:book},
these spaces are interpolation spaces and the Hausdorff-Young
inequality holds:
\begin{equation}\label{EQ:HY}
\|\widehat{f}\|_{\ell^{p'}(\Gh_{0})}\leq \|f\|_{L^{p}(G/K)},\;
\|f\|_{L^{p'}(G/K)}\leq \|\widehat{f}\|_{\ell^{p}(\Gh_{0})},
\quad
1\leq p\leq 2,
\end{equation}
where here and in the sequel we define the dual index $p'$ by $\frac{1}{p'}+\frac1p=1.$

We will need the following embedding property of these spaces:
\begin{equation}\label{EQ:cpt-lp-lq}
\ell^p(\Gh_{0})\hookrightarrow \ell^q(\Gh_{0})
\textrm{ and } \|\sigma\|_{\ell^q(\Gh_{0})}\leq \|\sigma\|_{\ell^p(\Gh_{0})},
\qquad  1\leq p\leq q\leq\infty.
\end{equation}
Indeed, we can assume $p<q$. Then, in the case $1\leq p<\infty$ and $q=\infty$, we can estimate
$$
\|\sigma\|_{\ell^\infty(\Gh_{0})}^p
 =  \p{\sup_{[\xi]\in\Gh_{0}} k_\xi^{-\frac12} \|\sigma(\xi)\|_\HS}^p
 \leq
\sum_{[\xi]\in\Gh_{0}} d_{\xi }k_\xi^{1-\frac p2} \|\sigma(\xi)\|_\HS^p = \|\sigma\|_{\ell^p(\Gh_{0})}^p.
$$
Let now $1\leq p<q<\infty$. Denoting $a_\xi:=d_{\xi}^{\frac1q}
k_\xi^{\frac 1q-\frac12}\|\sigma(\xi)\|_\HS$, we get
\begin{eqnarray*}
\|\sigma\|_{\ell^q(\Gh_{0})}
=
\p{\sum_{{[\xi]\in\Gh_{0}}} a_\xi^{q}}^{\frac 1q}
\leq
\p{\sum_{{[\xi]\in\Gh_{0}}} a_\xi^{p}}^{\frac 1p} =
\p{\sum_{{[\xi]\in\Gh_{0}}}  d_\xi^{\frac pq} k_{\xi}^{\frac pq-\frac p2}\|\sigma(\pi)\|_\HS^p}^{\frac1p}
\leq
\|\sigma\|_{\ell^p(\Gh_{0})},
\end{eqnarray*}
completing the proof.

Let $\Lap_{G/K}$ be the differential operator on $G/K$ obtained by $\Lap_{G}$
acting on functions that are constant on right cosets of $G$, i.e.,
such that $\widetilde{\Lap_{G/K}f}=\Lap_{G}\widetilde{f}$ for $f\in C^{\infty}(G/K)$.
For $[\xi]\in\Gh_{0}$, we denote by $\jp{\xi}$ the eigenvalue of
the operator $(1-\Lap_{G/K})^{1/2}$ corresponding to $\xi(x)_{ij}, 1\leq i\leq d_{\xi},
1\leq j\leq k_{\xi}$.
The operator $(1-\Lap_{G/K})^{1/2}$ is
a classical first order elliptic pseudo-differential operator
(see, e.g., Seeley \cite{Seeley}), with the same eigenspaces as the
operator $-\Lap_{G/K}$.
We refer to \cite{rt:book,rt:groups} for further details of the Fourier and
operator analysis on
compact homogeneous spaces.

Let $N(L)$ be the Weyl counting function for the elliptic
pseudo-differential operator
$(1-\Lap_{G/K})^{1/2}$, denoting the number of eigenvalues
$\leq L$,
counted with multiplicity.
From the above, we have
\begin{equation}\label{EQ:NL-sum}
N(L)=\mathop{\sum_{\jp{\xi}\leq L}}_{[\xi]\in\Gh_{0}} d_{\xi} k_{\xi}.
\end{equation}
For sufficiently
large $L$,
the Weyl asymptotic formula asserts that
\begin{equation}\label{EQ:NL}
N(L)\sim C_{0} L^{n}, \quad C_{0}=(2\pi)^{-n}\int_{\sigma_{1}(x,\omega)< 1} \dd x\dd\omega,
\end{equation}
where $n=\dim G/K$, and the integral is taken with respect to the
measure on the cotangent bundle $T^{*}(G/K)$ induced by the
canonical symplectic structure, with $\sigma_{1}$
being the principle symbol of the operator $(1-\Lap_{G/K})^{1/2},$
see, e.g., Shubin \cite{shubin:r}.

In the sequel we will often use the counting function for eigenvalues in the
dyadic annuli, and in these cases we can always assume that $\jp{\xi}$ is sufficiently
large, so that \eqref{EQ:NL-sum}--\eqref{EQ:NL} imply
\begin{equation}\label{EQ:NL1}
\sumxip d_{\xi}k_{\xi} \asymp 2^{sn}
\end{equation}
with a constant independent of $s$, in view of the estimates for the remainder
in Weyl spectral asymptotics, see, e.g.,
\cite{Duistermaat-Guillemin:IM-1975}.
This property will be often used in the sequel, and we may often write $\leq$ instead of
$\lesssim$ in the corresponding estimates to emphasise that the
constant is independent of $s$.

\section{General Nikolskii inequality}
\label{SEC:Nikolskii}

In this section we establish the Nikolskii inequality for trigonometric
functions on $G/K$. These are functions $T\in C^{\infty}(G/K)$ for which
only finitely many Fourier coefficients are non-zero, so that the
Fourier series \eqref{EQ:FSh-mat} is finite.
We will also use the notation $T_{L}$ for trigonometric polynomials
for which $\widehat{T_{L}}(\xi)=0$ for $\jp{\xi}>L$.
Thus, given the coefficients $\sigma=(\sigma({\xi}))_{[\xi]\in\Gh_{0}}\in\Sigma(G/K)$,
we can define the trigonometric polynomial with these Fourier coefficients by
\begin{equation}\label{EQ:TL}
T_{L }(x)=\sum_{\jp{\xi}\leq L} d_{\xi} \Tr (\sigma({\xi}) \xi(x)),
\end{equation}
where the sum is taken over all $[\xi]\in\Gh_{0}$ with $\jp{\xi}\leq L$.
In this case we have $\widehat{T_{L}}(\xi)=\sigma({\xi})$ for $\jp{\xi}\leq L$
and  $\widehat{T_{L}}(\xi)=0$ for $\jp{\xi}>L$.

Nikolskii's inequality in this setting has been analysed by
Pesenson \cite{Pesenson:Besov-2008}.
Here we give a more elementary proof of the Nikolskii inequality for $T_{L}$,
extending the range of indices $p,q$, and also
with the constant
interpreted in terms of the eigenvalue counting function
$N(L)$ in \eqref{EQ:NL}.

Let $D\in C^{\infty}(G/K)$ be a Dirichlet-type kernel, defined by setting its Fourier
coefficients to be
\begin{equation}\label{diri}
\widehat{D}(\xi):=\left(
\begin{matrix}
I_{k_{\xi}} & 0 \\
0 & 0
\end{matrix}
\right) \; \textrm{ for } [\xi]\in\Gh_{0} \textrm{ and } \jp{\xi}\leq L,
\end{equation}
and zero otherwise, where $I_{k_{\xi}}\in \mathbb{C}^{{k_{\xi}}\times {k_{\xi}}}$ is the
unit matrix.

\begin{thm}\label{THM2}
Let $M=G/K$ be a compact homogeneous space of dimension $n$.
Let $0<p<q\leq\infty$.
For $0<p\leq 2$ set $\rho:=1$, and for $2<p<\infty$, set
$\rho$ to be the smallest integer $\geq p/2$.
Then for any $L$ we have the estimate
\begin{equation}\label{EQ:N}
\|T_{L}\|_{L^{q}(G/K)}\leq N(\rho L)^{\frac1p-\frac1q}\|T_{L}\|_{L^{p}(G/K)},
\end{equation}
where $N(L)$ is the Weyl eigenvalue counting function for the elliptic
pseudo-differential operator
$(1-\Lap_{G/K})^{1/2}$.
Consequently, using \eqref{EQ:NL}, we have
\begin{equation}\label{EQ:N2}
\|T_{L}\|_{L^{q}(G/K)}\leq (C_{1}{\rho}^{n})^{\frac1p-\frac1q} L^{n(\frac1p-\frac1q)}
\|T_{L}\|_{L^{p}(G/K)},
\end{equation}
which holds for sufficiently
large $L$ for any constant
$C_{1}>C_{0}$, with $C_{0}$ given in  \eqref{EQ:NL}.

Moreover, the inequality \eqref{EQ:N}
is sharp for $p=2$ and $q=\infty$, and it becomes equality for $T=D$,
where $D$ is the Dirichlet-type kernel.
\end{thm}

\begin{remark}
For example, in the case of the real spheres $M={\mathbb S}^{n}$,
for $0<p\leq 2$ and $p<q\leq\infty$, we have
\begin{equation}\label{EQ:NS}
\|T_{L}\|_{L^{q}(\mathbb S^{n})}\leq C_{1}^{\frac1p-\frac1q}
L^{n(\frac1p-\frac1q)}
\|T_{L}\|_{L^{p}(\mathbb S^{n})},
\end{equation}
for any
$C_{1}>\frac{2}{n!},$ for all sufficiently large $L$.
Here the value of $C_{0}$ in \eqref{EQ:NL} follows from the explicit
formulae for the Weyl counting function on $\mathbb S^{n}$, see, e.g.,
Shubin \cite[Sec. 22]{shubin:r}. In particular, considering $\mathbb S^{1}=\mathbb T^{1}$ we obtain the constant
$C_1^{\frac1p-\frac1q}$, where $C_1>2$.
\end{remark}

\begin{proof}[Proof of Theorem \ref{THM2}]
We first note that formula \eqref{EQ:N2} follows from \eqref{EQ:N} by the asymptotics of
$N(L)$ in \eqref{EQ:NL}, so it is sufficient to prove  \eqref{EQ:N}.
We will give the proof of \eqref{EQ:N} and its sharpness in five  steps.

\medskip
\noindent
{\underline{Step 1.}
 The case $p=2$ and $q=\infty$.}
From formula \eqref{EQ:TL} and using
$\|\xi(x)\|_{\HS}=k_{\xi}^{1/2}$ from \eqref{EQ:xi-HS},
we can estimate
\begin{equation}\label{EQ2inf}
\begin{split}
\|T_{L}\|_{L^{\infty}(G/K)}&\le \mathop{\sum_{\jp{\xi}\leq L}}_{[\xi]\in\Gh_{0}}
d_{\xi} \|\sigma({\xi})\|_{\HS}
\|\xi(x)\|_{\HS}\\
&=\mathop{\sum_{\jp{\xi}\leq L}}_{[\xi]\in\Gh_{0}}  d_{\xi}k_{\xi}^{1/2} \|\sigma({\xi})\|_{\HS}\\
& \le
\left( \mathop{\sum_{\jp{\xi}\leq L}}_{[\xi]\in\Gh_{0}} d_{\xi} k_{\xi}\right)^{1/2}
\left( \mathop{\sum_{\jp{\xi}\leq L}}_{[\xi]\in\Gh_{0}}
d_{\xi}\|\sigma({\xi})\|_{\HS}^{2}\right)^{1/2} \\
& =
N(L)^{1/2} \|T_{L}\|_{L^{2}(G/K)},
\end{split}
\end{equation}
where in the last inequality we used \eqref{EQ:NL-sum} and the Plancherel
identity.

\medskip
\noindent
{\underline{Step 2.}
 The case $p=2$ and $2<q\leq\infty$.}
 We take
$1\leq q'<2$ so that
$\frac1q+\frac{1}{q'}=1.$
We set $r:=\frac2{q'}$ so that its dual index is
satisfies $\frac{1}{r'}=1-\frac{q'}{2}.$
By the Hausdorff-Young inequality in \eqref{EQ:HY},
and using \eqref{EQ:Lp-sigmaGK} and H\"older's inequality, we obtain
\[
\begin{split}
\|T_{L}\|_{L^{q}(G/K)}&\le \|\sigma\|_{\ell^{q'}(\Gh_{0})}=
\left(\sum_{\jp{\xi}\leq L} d_{\xi}^{1-\frac{q'}{2}} k_{\xi}^{1-\frac{q'}{2}} d_{\xi}^{{\frac{q'}{2}}}
\|\sigma({\xi})\|_{\HS}^{q'}
\right)^{\frac{1}{q'}}\\
& \le
\left( \sum_{\jp{\xi}\leq L} (d_{\xi}k_{\xi})^{(1-\frac{q'}{2})r'}\right)^{\frac{1}{q'r'}}
\left( \sum_{\jp{\xi}\leq L} d_{\xi}^{\frac{q'r}{2}}\|\sigma({\xi})\|_{\HS}^{q'r}\right)^{\frac{1}{q'r}} \\
& =
\left( \sum_{\jp{\xi}\leq L} d_{\xi}k_{\xi}\right)^{\frac{1}{q'}-\frac12}
\left( \sum_{\jp{\xi}\leq L} d_{\xi}\|\sigma({\xi})\|_{\HS}^{2}\right)^{\frac{1}{2}} \\
& =
N(L)^{\frac12-\frac1q} \|T_{L}\|_{L^{2}(G/K)},
\end{split}
\]
where we have used that
$\frac{q'r}{2}=1$.

\medskip
\noindent
{\underline{Step 3.}
 The case $p>2$.}
For $2<p<q\le \infty$ and an integer $\rho\ge p/2$, we claim to have
$$
\|T_{L}\|_{L^{q}(G/K)}\le N( \rho L)^{(1/p-1/q)} \|T_{L}\|_{L^{p}(G/K)}.
$$
Indeed, if $q=\infty$, for $T_{L}\not\equiv 0$, we get
\begin{equation}\label{EQ:df}
\begin{split}
\|T_{L}^\rho\|_{L^{2}}
&=
\||T_{L}|^{\rho-p/2} |T_{L}|^{p/2}\|_{L^{2}}
\le
\||T_{L}|^{\rho-p/2}\|_{L^{\infty}}
\||T_{L}|^{p/2}\|_{L^{2}}
\\
&=
\|T_{L}\|_{L^{\infty}}^{\rho-p/2}
\||T_{L}|^{p/2}\|_{L^{2}}
=
\|T_{L}\|_{L^{\infty}}^{\rho} \|T_{L}\|_{L^{\infty}}^{-p/2}
\||T_{L}|^{p/2}\|_{L^{2}} \\
&
=
\|T_{L}^{\rho}\|_{L^{\infty}} \|T_{L}\|^{-p/2}_{L^{\infty}}
\|T_{L}\|_{L^{p}}^{p/2}
\\
& \le
N( \rho L)^{1/2}
\|T_{L}^{\rho}\|_{L^{2}} \|T_{L}\|^{-p/2}_{L^{\infty}} \|T_{L}\|_{L^{p}}^{p/2},
\end{split}
\end{equation}
where we have used \eqref{EQ2inf} in the last line. We have also
used the fact that $T_{L}^{\rho}$ is a trigonometric polynomial of degree
$\rho L$, which follows from the decomposition of (Kronecker's)
tensor products of representations
into irreducible components by looking at the corresponding characters
on the maximal torus of $G$.
Therefore, using that $T_{L}^{\rho}\not\equiv 0$, we have
\begin{equation}\label{EQ:auxpinf}
\begin{split}
\|T_{L}\|_{L^{\infty}}
& \le
N( \rho L)^{1/p}
\|T_{L}\|_{L^{p}}.
\end{split}
\end{equation}
For $p<q<\infty$ we obtain
\begin{equation}\label{EQ:df1}
\begin{split}
\|T_{L}\|_{L^{q}}&= \||T_{L}|^{1-p/q} |T_{L}|^{p/q}\|_{L^{q}}\le
\|T_{L}\|_{L^{\infty}}^{1-p/q} \|T_{L}\|_{L^{p}}^{p/q}
\\
& \le
N( \rho L)^{1/p(1-p/q)} \|T_{L}\|_{L^{p}}^{1-p/q}
\|T_{L}\|_{L^{p}}^{p/q},
\end{split}
\end{equation}
where we have used \eqref{EQ:auxpinf},
which implies \eqref{EQ:N} in this case.

\medskip
\noindent
{\underline{Step 4.}
The case $0<p<2$.}
For $p<q\le\infty$, $0<p\le 2$,
proceding similar to (\ref{EQ:df}) with $\rho=1$ (note that $p/2\le 1$), we get
\[
\begin{split}
\|T_{L}\|_{L^{\infty}}
& \le
N( L)^{1/p}
\|T_{L}\|_{L^{p}}.
\end{split}
\]
Then the estimate as in
(\ref{EQ:df1}) implies \eqref{EQ:N} also in the case $p<q\le\infty$, $0<p\le 2$.

\medskip
\noindent
{\underline{Step 5.}
Sharpness.}
If $D$ is the  Dirichlet-type kernel \eqref{diri},
then using Plancherel's identity and the definition of $N(L)$ we can calculate
\begin{equation}\label{EQ:Dirichlet-L2}
\|D\|_{L^{2}(G/K)}=\p{\sum_{[\xi]\in\Gh_{0}} d_{\xi}
\|\widehat{D}(\xi)\|_{\HS}^{2}}^{1/2}=
\p{\sum_{[\xi]\in\Gh_{0}} d_{\xi}k_{\xi}}^{1/2}=N(L)^{1/2}.
\end{equation}
On the other hand, writing the Fourier series
$D(x)=\sum_{[\xi]\in\Gh_{0}} d_{\xi}
\Tr\p{\xi(x)\widehat{D}(\xi)}$, and recalling our convention about zeros in the
last rows of $\xi(x)$, we have
$$
D(eK)=\sum_{[\xi]\in\Gh_{0}} d_{\xi}\Tr(\widehat{D}(\xi))=
\sum_{\jp{\xi}\leq L} d_{\xi} k_{\xi}=N(L).
$$
Combining this with \eqref{EQ:Dirichlet-L2} in the Nikolskii inequality \eqref{EQ:N} we obtain
$$
N(L)=D(eK)\leq \|D\|_{L^{\infty}}\leq N(L)^{1/2}\|D\|_{L^{2}}=N(L),
$$
showing the sharpness of the constant in \eqref{EQ:N} in the case
$p=2$ and $q=\infty$.
\end{proof}

\section{Nikolskii inequalities on groups}
\label{SEC:Nikolskii1}

In this section we prove that, in the case of the group $G$,
 for a given trigonometric polynomial $T$,
a constant in the Nikolskii inequality depends on the number of non-zero Fourier coefficients
of $T$, and this statement is sharp for $p=2$ and $q=\infty$.

\begin{lem}\label{PROP:DK}
Let $1\leq p<q\leq\infty$ be such that $\frac 1p \geq \frac 1q + \frac12.$
Let $T$ be a trigonometric polynomial on the compact Lie group $G$.
Then
\begin{equation}\label{EQ:Nik-refined}
\|T\|_{L^q(G)}\leq \p{\sum_{\widehat{T}(\xi)\not=0} d_\xi^2}^{\frac 1p-\frac 1q} \|T\|_{L^p(G)}.
\end{equation}
For $p=2$ and $q=\infty$, this estimate is sharp.
\end{lem}
\begin{proof}
Let us define the Dirichlet-type kernel $D$ by setting its Fourier coefficients
to be $\widehat{D}(\xi):=I_{\dxi}$ if $\widehat{T}(\xi)\not=0$, and
$\widehat{D}(\xi):=0$ if $\widehat{T}(\xi)=0$.
Then for any $1\leq r'<\infty$, we can calculate
$$
\|\widehat{D}\|_{\ell^{r'}(\Gh)}^{r'}=\sum_{\widehat{T}(\xi)\not=0} d_\xi^{r'(\frac{2}{r'}-\frac12)}
\|\widehat{D}(\xi)\|_\HS^{r'}=\sum_{\widehat{T}(\xi)\not=0} d_\xi^2.
$$
Consequently, for any $r\geq 2$, by the Hausdorff-Young inequality
\eqref{EQ:HY} with $k_{\xi}=d_{\xi}$, we get
$$
\|D\|_{L^r(G)}\leq \|\widehat{D}\|_{\ell^{r'}(\Gh)}
= \p{\sum_{\widehat{T}(\xi)\not=0} d_\xi^2}^{\frac{1}{r'}}.
$$
Now, we take $r$ so that $\frac1p-\frac1q=1-\frac1r=\frac{1}{r'}$, and observe that
the condition $\frac 1p \geq \frac 1q + \frac12$ implies $r\geq 2$.
On the other hand, we have $\widehat{T*D}=\widehat{D}\widehat{T}=\widehat{T}$, so that
$T=T*D$, see, e.g., \cite[Proposition 7.7.5]{rt:book}.
 Applying the Young inequality and using the estimates for $D$ as above, we obtain
\begin{equation}\label{EQ:Young-ineq}
\|T\|_{L^q(G)}=\|T*D\|_{L^q(G)}\leq \|T\|_{L^p(G)}\|D\|_{L^r(G)}\leq
\|T\|_{L^p(G)}\p{\sum_{\widehat{T}(\xi)\not=0} d_\xi^2}^{\frac{1}{r'}},
\end{equation}
implying \eqref{EQ:Nik-refined}. The sharpness follows by an argument similar to
that in the proof of Theorem \ref{THM2}.
\end{proof}
The main result of this section is the following

\begin{thm}\label{THM3}
Let $0<p<q\leq\infty$. For $0<p\leq 2$ set $\rho:=1$, and for $2<p<\infty$, set $\rho$ to be the smallest integer $\geq p/2$.
Let $T$ be a trigonometric polynomial on the compact Lie group $G$.
Then
\begin{equation}\label{EQ:Nik-refined-1}
\|T\|_{L^q(G)}\leq \p{\sum_{\widehat{T^\rho}(\xi)\not=0} d_\xi^2}^{\frac 1p-\frac 1q} \|T\|_{L^p(G)}.
\end{equation}
Moreover, this inequality is sharp for $p=2$ and $q=\infty$ and it becomes equality for $T=D$, where $D$ is the Dirichlet-type kernel.
\end{thm}
\begin{remark}
 We note that if $T=T_{L}$, i.e., if $\widehat{T}(\xi)=0$ for $\jp{\xi}>L$,
we have $\sum_{\widehat{T}(\xi)\not=0} d_\xi^2\leq N(L)$, uniformly over such $T$,
in agreement with the corresponding part of Theorem \ref{THM2}.
\end{remark}

\begin{proof}[Proof of Theorem \ref{THM3}]
From Lemma \ref{PROP:DK} we
$$
\|T\|_{L^\infty(G)}\leq \p{\sum_{\widehat{T}(\xi)\not=0} d_\xi^2}^{\frac 12} \|T\|_{L^2(G)}.
$$
Then following the proof of Theorem \ref{THM2},
we get (\ref{EQ:Nik-refined-1}) for $2\le p <q\le \infty$ and  $0<p\le 2$, $p\le q\le \infty$.
The sharpness follows by an argument similar to
that in the proof of Theorem \ref{THM2}.
\end{proof}

\medskip
From (the proof of) Lemma \ref{PROP:DK}, for a function $f\in L^{p}(G)$, we immediately get the following
estimate for the partial sums of its Fourier series:
if
$$
S_{L}f(x):=\sum_{\jp{\xi}\leq L} d_{\xi}\ \Tr(\whf(\xi)\xi(x)),
$$
 we have
\begin{equation}\label{EQ:SLest-1}
\|S_{L}f\|_{L^q(G)}\leq N(L)^{\frac 1p-\frac 1q} \|f\|_{L^p(G)},\qquad \frac 1p \geq \frac 1q + \frac12.
\end{equation}
Indeed,
defining the Dirichlet-type kernel $D$ by setting its Fourier coefficients
to be $\widehat{D}(\xi):=I_{\dxi}$ for $\jp{\xi}\leq L$, and
$\widehat{D}(\xi):=0$ for $\jp{\xi}>L$, from the identity
$\widehat{S_{L}f}=\widehat{D}\widehat{f}$, we get
$S_{L}=f*D$, so that applying the Young inequality with indices as in
\eqref{EQ:Young-ineq} and arguing as in the proof of Lemma \ref{PROP:DK}, we obtain
\eqref{EQ:SLest-1} since $k_{\xi}=d_{\xi}$ in this case.


But in fact we can prove a sharper estimate:
\begin{cor}\label{THM3-2}
Let $G$ be a compact Lie group and let
$1\leq p<q\leq\infty$ be such that $\frac 1p > \frac 1q + \frac12.$
For $f\in L^{p}(G)$ we have
\begin{equation}\label{EQ:SLest}
\left(\sum_{k=1}^\infty\frac{\left(k^{1-1/p+1/q}\sup_{N(L)\geq k}\frac1{N(L)}\|S_{L}f\|_{L^q(G)}\right)^p}k\right)^{1/p}\leq C \|f\|_{L^p(G)}.
\end{equation}
In particular, we have
\begin{equation}\label{EQ:SLest2}
{N(L)}^{\frac 1q-\frac 1p}\|S_{L}f\|_{L^q(G)} =o(1)\quad \mbox{as} \qquad L\to \infty
\end{equation}
and therefore
\begin{equation}\label{EQ:SLest4}
L^{n(\frac 1q-\frac 1p)}
\|S_{L}f\|_{L^q(G)} =o(1)\quad \mbox{as} \qquad L\to \infty,
\end{equation}
with $n=\dim G$.
\end{cor}

\begin{proof}
We have shown in (\ref{EQ:SLest-1}) that
\begin{equation}\label{EQ:SLest3}
N(L)^{-\frac 1p+\frac 1q} \|S_{L}f\|_{L^q(G)}\leq  \|f\|_{L^p(G)}.
\end{equation}
Let us define a quasi-linear operator  $A$ by
 $$
 Af := \left\{\sup_{N(L)\geq k}\frac1{N(L)}\|S_{L}f\|_{L^q(G)}\right\}_{k=1}^\infty.
 $$
Then \eqref{EQ:SLest3} implies that  $A$ is bounded from  $L_p(G)$ to $l_{r,\infty}$, where $1/r=1-1/p+1/q$
and 
 $l_{r,\infty}$ denotes the  Lorentz sequence space.
 Indeed,
$$
\|Af\|_{l_{r,\infty}}=\sup_k k^{1/r}\sup_{N(L)>k}\frac1{N(L)}\|S_{L}f\|_{L^q(G)}\leq \sup_k\sup_{N(L)>k}\frac{N(L)^{1/r}}{N(L)}\|S_{L}f\|_{L^q(G)}
\leq \|f\|_{L_p}.
$$
Let $(p_0,r_0)$ and $(p_1,r_1)$ be such that $$p_0<p<p_1, \quad r_0<r<r_1,\quad \frac{1}{r_0}+\frac{1}{p_0}=
\frac{1}{r_1}+\frac{1}{p_1}=\frac{1}{r}+\frac{1}{p}.$$
Then, using the interpolation  theorem
(see, e.g., \cite{Bergh-Lofstrom:BK-interpolation-spaces-1976}),
since
$A: L_{p_0}(G) \to l_{r_0,\infty}$ and $A: L_{p_1}(G) \to l_{r_1,\infty}$ are
bounded, we get that
$$A: L_{p}(G) \to l_{r,p}$$
is bounded, which is \eqref{EQ:SLest}.

Now,
for any $\xi\in \mathbb{N}$ we choose
$t\in \mathbb{N}:$ $2^t\le \xi^n < 2^{t+1}$. Then \eqref{EQ:SLest} implies
$$
\sum_{k=2^{t-1}}^{2^{t}-1} \frac{\left(k^{1-1/p+1/q}\sup_{N(L)\geq k}\frac1{N(L)}\|S_{L}f\|_{L^q(G)}\right)^p}k\to 0.
$$
The left-hand side is greater than
\begin{align*}
&
C
\left(2^{t(1-1/p+1/q)}
\sup_{N(L)\geq 2^{t}}\frac1{N(L)}  \|S_{L}f\|_{L^q(G)} \right)^p
\sum_{k=2^{t-1}}^{2^{t}-1} \frac{1}k
\\
&\ge
C\left(2^{t(1-1/p+1/q)}
\sup_{L^n\geq 2^t} \frac1{L^n}  \|S_{L}f\|_{L^q(G)} \right)^p
\ge C
\left(\frac{2^{t(1-1/p+1/q)}} {\xi^n}  \|S_{\xi}f\|_{L^q(G)}\right)^p
\\&
\ge C
\left(
{N(\xi)}^{1/q-1/p}\|S_{\xi}f\|_{L^q(G)}\right)^p,
\end{align*}
finishing the proof of \eqref{EQ:SLest2}.

Finally, \eqref{EQ:SLest4} follows from \eqref{EQ:SLest2}, using \eqref{EQ:NL}
\end{proof}
Note that a similar result for periodic functions can be found in
\cite{Bergh-Lofstrom:BK-interpolation-spaces-1976,Nursultanov:Steklov-2006}.
Also, for certain values of $p$, there are low bounds results for estimates in Corollary \ref{THM3-2}, showing that 
norms of some characters (depending on $L$)
should still go to infinity, see, e.g., \cite[Corollary 4]{Giulini-Soard-Travaglini:JFA-1982}.

\section{Besov spaces}
\label{SEC:Besov}
In this section we analyse embedding properties of Besov spaces on compact
homogeneous spaces $G/K$. By using the Fourier series \eqref{EQ:FSh-mat},
we define the Besov space
\begin{equation}\label{EQ:Besov1}
B^{r}_{p,q}=B^{r}_{p,q}(G/K)=\left\{f\in\Dcal'(G/K):\,\,\|f\|_{B^{r}_{p,q}}<\infty\right\},
\end{equation}
where
\begin{equation}\label{EQ:Besov2}
 \|f\|_{B^{r}_{p,q}} :=
 \left(\sum_{s=0}^\infty 2^{sr q} \Big\|\sum_{2^{s}\le\jp{\xi}<2^{s+1}}
 \dxi\ \Tr\p{\whf(\xi)\xi(x)}\Big\|_p^q\right)^{1/q}, \qquad q<\infty,
\end{equation}
and
\begin{equation}\label{EQ:Besov3}
 \|f\|_{B^{r}_{p,q}} :=
 \sup\limits_{s\in\mathbb{N}}
2^{sr } \Big\|\sum_{2^{s}\le\jp{\xi}<2^{s+1}}
 \dxi\ \Tr\p{\whf(\xi)\xi(x)}\Big\|_p, \qquad q=\infty.
\end{equation}
Here we allow $r\in\R$ and $0<p,q\leq\infty$. We also note that since we always have
$\jp{\xi}\geq 1$, the trivial
representation is included in this norm.

We first analyse these  Besov spaces using the global Littlewood-Paley theory, and
in Section \ref{SEC:localisations} we show that for certain ranges of indices
these spaces agree with the Besov spaces that could be defined on $G/K$
as a manifold, using the standard Besov spaces on $\Rn$ in local coordinates.

On unimodular Lie groups Besov spaces have been analysed in
\cite{Skr} in terms of the heat kernel, however no embedding theorems have been
proved. In addition,
using the Littlewood-Paley decomposition, one can also use the characterisation
given in \cite{FMV}. Using the Nikolskii inequality and the Fourier analysis in
Section \ref{SEC:Fourier}, we can establish embedding properties
for these spaces.

First, let us prove that the norm $ \|f\|_{B^{r}_{p,q}}$ is equivalent  to a certain
approximative characteristic of $f$.
\begin{prop}\label{proposit} Let $0<p\le\infty$, $0<q\leq\infty$, and $r\in\R$. We have
$$\|f\|_{B^{r}_{p,q}}\asymp \left|\int_{G/K} f(x) dx\right|+
 \left(\sum_{s=0}^\infty 2^{sr q} \Big\| f- S_{2^{s}}f
 \Big\|_p^q
 \right)^{1/q},
$$
where $S_{L}f$ is the  partial Fourier series of $f$, that is,
$S_{L}f(x)=\sum_{\jp{\xi}\leq L} d_{\xi}\ \Tr(\whf(\xi)\xi(x)).$
\end{prop}

\begin{proof}
First we observe that if $1$ is the trivial representation of $G$, we have
$|S_{1}f|=|\widehat{f}(1)|=|\int_{G/K} f(x) dx|$,
and that without loss of generality we can assume that
$\widehat{f}(1)=0$.
Denote
$$a_s:= S_{2^{s+1}}f(x)-
S_{2^{s}}f(x)=\sum_{2^{s}<\jp{\xi}\le 2^{s+1}}
 \dxi\ \Tr\p{\whf(\xi)\xi(x)}.$$
 With this notation we can write
$$
\|f\|_{B^{r}_{p,q}} \asymp
 \left(\sum_{s=0}^\infty 2^{sr q} \big\| a_s\big\|_p^q\right)^{1/q}.
$$
We first show ``$\gtrsim$".
 If $0<q<1$, then
\begin{eqnarray*}
 \left(\sum_{s=0}^\infty 2^{sr q} \Big\| f- S_{2^{s}}f \Big\|_p^q \right)^{1/q}
 &\asymp&
  \left(\sum_{s=0}^\infty 2^{sr q} \Big\| \sum_{k=s}^\infty a_k \Big\|_p^q \right)^{1/q}
  \\
 &\le&
  \left(\sum_{s=0}^\infty 2^{sr q}  \sum_{k=s}^\infty \big\|a_k \big\|_p^q \right)^{1/q}
  \\
 &\le&
  \left(\sum_{k=0}^\infty \big\|a_k \big\|_p^q \sum_{s=0}^k 2^{sr q}   \right)^{1/q}
  \\
& \lesssim  &
  \left(\sum_{k=0}^\infty 2^{kr q}  \big\|a_k \big\|_p^q   \right)^{1/q}
  \asymp\|f\|_{B^{r}_{p,q}}.
\end{eqnarray*}
If $q\ge 1$, using Hardy's inequalities, we also get
\begin{eqnarray*}
 \left(\sum_{s=0}^\infty 2^{sr q} \Big\| f- S_{2^{s}}f \Big\|_p^q \right)^{1/q}
 &\le&
  \left(\sum_{s=0}^\infty 2^{sr q} \left( \sum_{k=s}^\infty \big\|a_k \big\|_p\right)^q \right)^{1/q}
  \\
 &\lesssim&
  \left(\sum_{k=0}^\infty 2^{kr q}  \big\|a_k \big\|_p^q   \right)^{1/q}
  \asymp\|f\|_{B^{r}_{p,q}}.
\end{eqnarray*}
Indeed, if $q\ge 1,$ $\varepsilon>0$ and  $\alpha_k\ge 0 $, the Hardy inequality asserts that
$$\sum\limits_{s=0}^\infty 2^{\varepsilon s}\left(\sum_{k =s}^\infty \alpha_k\right)^q \le C(q, \varepsilon)\sum\limits_{s=0}^\infty 2^{\varepsilon s} \alpha_s^q.
$$
To prove the part ``$\lesssim$", we write $a_k=a_k+ f -f$
to obtain
\begin{eqnarray*}
\|f\|_{B^{r}_{p,q}}
 &\asymp&
  \left(\sum_{k=0}^\infty 2^{kr q}  \big\|a_k \big\|_p^q   \right)^{1/q}
  \lesssim
 \left(\sum_{s=0}^\infty 2^{sr q} \Big\| f- S_{2^{s}}f \Big\|_p^q \right)^{1/q}
 ,
\end{eqnarray*}
completing the proof.
\end{proof}

For $r\in\R$, we denote by $H_{p}^{r}$ the Sobolev space on $G/K$ defined in
local coordinates, i.e. the space of distributions such that in each local
coordinate systems they belong to the usual Sobolev spaces
 $H_{p}^{r}(\Rn)$. By ellipticity it can be described as the set of all $f\in\Dcal'(G/K)$ such that
we have $(1-\Lap_{G/K})^{r/2}f\in L_{p}(G/K)$. Writing the Fourier series for the
lifting of $f$ to $G$, we see from \eqref{EQ:FSh-mat} that
the Fourier series of $(1-\Lap_{G/K})^{r/2}f$ is
given by $\sum_{[\xi]\in\Gh_{0}} \dxi\ \jp{\xi}^{r}\Tr\p{\whf(\xi)\xi(x)}$,
and hence we have
\begin{equation}\label{EQ:Sobolev-norms}
\|f\|_{ H_{p}^{r}}
\asymp
\left\|
\sum_{[\xi]\in\Gh_{0}} \dxi\ \jp{\xi}^{r}\Tr\p{\whf(\xi)\xi(x)}   \right\|_{p}.
\end{equation}
We will often use Plancherel's identity in the following form:
\begin{equation}\label{EQ:Plancherel-partial}
\left\|\sumxip \dxi\ \Tr\p{\whf(\xi)\xi(x)}\right\|_{2}=\p{\sumxip \dxi \whfhs^{2}}^{1/2},
\end{equation}
which holds for $L^{2}(G/K)$ if we use our convention of having zeros in
$\xi(x)$, the equality
$\|\widetilde{f}\|_{L^{2}(G)}=\|f\|_{L^{2}(G/K)}$ in our choice of normalisation of
measures, and apply the Plancherel's identity on $G$.

\medskip
We now collect the embedding properties of $B_{p,q}^{r}$ in the following theorem.
\begin{thm}\label{THM:Besov}
Let $G/K$ be a compact homogeneous space of dimension $n$. Below, we allow
$r\in\R$ unless stated otherwise.
We have
\begin{itemize}

\item[(1)] \  \   $B_{p,q_1}^{r+\varepsilon}\hookrightarrow B_{p,q_1}^{r}\hookrightarrow B_{p,q_2}^{r}\hookrightarrow B_{p,\infty}^{r},
\qquad 0<\varepsilon,
\quad 0<p\leq\infty,\; 0<q_1\le q_2\le \infty;$

\item[(2)] \ \
$B_{p,q_1}^{r+\varepsilon}\hookrightarrow B_{p,q_2}^{r}, \qquad 0<\varepsilon, \quad
0<p\leq\infty,\; 1\leq q_2<q_1<\infty;$

\item[(3)] \  \
  $B_{2,2}^{r}=H^{r};$

\item[(4)] \ \
$H^{r+\varepsilon}\hookrightarrow B_{2,q}^{r}, \qquad \varepsilon, q>0;$

\item[(5)] \ \
$B_{p_1,q}^{r_1}\hookrightarrow B_{p_2,q}^{r_2}, \qquad
0<p_1\le p_2\leq\infty,\; 0<q<\infty,\; r_2= r_1 - n(\frac{1}{p_1}-\frac{1}{p_2});$

\item[(6)] \ \
$B_{p,p}^{r}\hookrightarrow H^r_p \hookrightarrow B_{p,2}^{r}, \qquad 1<p\le 2;$

\item[(7)] \ \
$B_{p,2}^{r}\hookrightarrow H^r_p \hookrightarrow B_{p,p}^{r}, \qquad 2\le p<\infty;$

\item[(8)] \ \
$B_{p,1}^r\hookrightarrow L_q, \qquad
0< p<q\le\infty, \; r=n(\frac1p-\frac1q);$

\item[(9)] \ \
$B_{p,q}^r\hookrightarrow L_q, \qquad
1<p<q<\infty, \; r=n(\frac1p-\frac1q)$.
\end{itemize}
\end{thm}
We note that (6) and (7) can be rewritten as
\begin{itemize}
\item[(6')] \ \
$B_{p,\min\{p,2\}}^{r}\hookrightarrow H^r_p \hookrightarrow B_{p,\max\{p,2\}}^{r},
\qquad 1<p<\infty.$
\end{itemize}

\begin{proof}[Proof of Theorem \ref{THM:Besov}]

\medskip
\noindent
(1) These embeddings follow from
$$\sup_s a_{s}\le(\sum_{s} a_{s}^{q_2})^{1/q_2}\le (\sum_{s} a_{s}^{q_1})^{1/q_1}
\le (\sum_{s} 2^{s\varepsilon q_1}a_{s}^{q_1})^{1/q_1}$$
for a sequence $a_{s}\ge 0$ and $\varepsilon>0$.

\medskip
\noindent
(2) Using $q_2<q_1$, by H\"{o}lder's inequality, we get
\begin{eqnarray*}
\|f\|_{B_{p,q_2}^{r}}
&\asymp&
\left(\sum_s \frac{2^{s q_2 (r+\varepsilon)}}{2^{s q_2 \varepsilon }} \left\|
\sumxip \dxi\ \Tr\p{\whf(\xi)\xi(x)}  \right\|_p^{q_2}\right)^{1/q_2}
\\&\le &
\left(\sum_s {2^{s q_1 (r+\varepsilon)}}  \left\|
\sumxip \dxi\ \Tr\p{\whf(\xi)\xi(x)}\right\|_p^{q_1}\right)^{1/q_1} \times \\
& & \qquad \times
\left(\sum_s \frac{1}{2^{sq_1q_2/(q_1-q_2) \varepsilon}} \right)^{(q_1-q_2)/(q_1q_2)}
\\&\le& C
\|f\|_{ B_{p,q_1}^{r+\varepsilon}}.
\end{eqnarray*}

\medskip
\noindent
(3)  Using Plancherel's identity \eqref{EQ:Plancherel-partial}
we get
\begin{eqnarray*}
\|f\|_{B_{2,2}^{r}} & = &
\left(
\sum_s
2^{2sr}
\left\|\sumxip  \dxi\ \Tr\p{\whf(\xi)\xi(x)}\right\|^2_{2}\right)^{1/2} \\
& \asymp &
\left( \sum_s \left(\sumxip \jp{\xi}^{2r} \dxi \whfhs^2\right)\right)^{1/2} \\
& \asymp &
\|\jp{\xi}^{r}\whf\|_{\ell^2(\Gh_{0})}
= \|f\|_{H^{r}}.
\end{eqnarray*}

\medskip
\noindent
(4) This embedding follows from properties (1)--(3).

\medskip
\noindent
(5) Using Nikolskii's inequality from Theorem \ref{THM2},
\begin{eqnarray*}
\|f\|_{ B_{p_2,q}^{r_2}}
&\asymp&
\left(\sum_s 2^{s q r_2} \left\|\sumxip \dxi\ \Tr\p{\whf(\xi)\xi(x)} \right\|_{p_2}^{q}\right)^{1/q}
\\&\le& C
\left(\sum_s 2^{s q r_2} 2^{s q n(\frac{1}{p_1}-\frac{1}{p_2})} \left\|
\sumxip \dxi\ \Tr\p{\whf(\xi)\xi(x)}  \right\|_{p_1}^{q}\right)^{1/q}
\asymp
\|f\|_{ B_{p_1,q}^{r_1}}.
\end{eqnarray*}

\medskip
\noindent
(6) First,  by the Littlewood-Paley theorem (see \cite{FMV}), applied to
functions which are constant on right cosets, we get
\begin{eqnarray}\label{EQ:Littlewood-Paley}
\|f\|_{ H_{p}^{r}}
&\asymp&
\left\|
\sum_s 2^{s r} \sumxip \dxi\ \Tr\p{\whf(\xi)\xi(x)}   \right\|_{p} 
\nonumber
\\&\asymp&
\left\|\left[\sum_s 2^{2s r} \left|\sumxip \dxi\ \Tr\p{\whf(\xi)\xi(x)} \right|^2\right]^{1/2} \right\|_{p}. 
\end{eqnarray}
Since we have already shown the case $p=2$ in (3), we can assume that $1<p<2$.
Using the inequality
$(\sum_k a_{k}^{2})^{1/2}\le (\sum_k a_{k}^{p})^{1/p}, a_{k}\ge 0$, we obtain
\begin{eqnarray*}
\|f\|_{ H_{p}^{r}}
&\lesssim&
\left\|\left[\sum_s 2^{p s r} \left | \sumxip
\dxi\ \Tr\p{\whf(\xi)\xi(x)}
\right|^p\right]^{1/p} \right\|_{p}
\\
&\leq&
\left(\sum_s 2^{ps r} \left\| \sumxip
\dxi\ \Tr\p{\whf(\xi)\xi(x)}  \right\|^p_p  \right)^{1/p}
=\|f\|_{ B_{p,p}^{r}}.
\end{eqnarray*}
On the other hand, using
Minkowski inequality $(\sum_j(\int f_j)^\alpha)^{1/\alpha}\le \int(\sum_j f_j^\alpha)^{1/\alpha}$
for $\alpha>1$ and $f_{j}\geq 0$,
we get for $\alpha=2/p$,
\begin{eqnarray*}
\|f\|_{ B_{p,2}^{r}}&=&
\left(\sum_s 2^{2 s r} \left\| \sumxip\dxi\ \Tr\p{\whf(\xi)\xi(x)}  \right\|^2_p  \right)^{1/2}
\\&\le&
\left(
\int\left[ \sum_s 2^{2 s r}  \left| \sumxip
\dxi\ \Tr\p{\whf(\xi)\xi(x)}  \right|^2 \right]^{p/2} dx \right)^{1/p}
\\&\le&
\left\|\left[\sum_s 2^{2s r} \left|\sumxip
\dxi\ \Tr\p{\whf(\xi)\xi(x)}  \right|^2\right]^{1/2} \right\|_{p} 
\asymp \|f\|_{ H_{p}^{r}},
\end{eqnarray*}
using \eqref{EQ:Littlewood-Paley} in the last equivalence of norms again.
The proof of (7) is similar.

\medskip
\noindent
{(8)}  Using the Fourier series representation \eqref{EQ:FSh-mat} and
Nikolskii's inequality from Theorem \ref{THM2},  we get
\begin{eqnarray*}
\|f\|_{L_q} & \leq& \sum_{s=0}^\infty \left\|\sumxip
\dxi\ \Tr\p{\whf(\xi)\xi(x)}\right\|_{L_q}\\
&\leq&\sum_{s=0}^\infty 2^{sn(\frac1p-\frac1q)}\left\|\sumxip
\dxi\ \Tr\p{\whf(\xi)\xi(x)}\right\|_{L_p}=\|f\|_{{B}_{p,1}^r}.
\end{eqnarray*}

\medskip
\noindent
{(9)} We show that (9) follows from (8).
Let $F$ be such that $F: B_{p,1}^{r} \hookrightarrow L_q$. Then for parameters
$p,\;q,\;r$ one can find couples  $(q_0,\;r_0),$ $(q_1,\;r_1)$ and $\theta\in (0,1)$ such that
$$
n\Big(\frac1p-\frac1{q_0}\Big)=r_0,\qquad
n\Big(\frac1p-\frac1{q_1}\Big)=r_1,
\qquad
r_0<r<r_1,
$$
and
$$r=(1-\theta)r_0+\theta r_1, \qquad\frac1q=\frac{1-\theta}{q_0}+\frac{\theta}{q_1}.$$
Since
$
F:\;B_{p,1}^{r_0}\rightarrow L_{q_0}
$
and
$
F:\;B_{p,1}^{r_1}\rightarrow L_{q_1},
$
then by the interpolation theorems
$$
F:\;B_{p,q}^r=\left(B_{p,1}^{r_0},\;\;B_{p,1}^{r_1}\right)_{\theta, q}\rightarrow \left(L_{q_0},\;L_{q_1}\right)_{\theta,q}=L_q,
$$
i.e.,  $B_{p,q}^{r}\hookrightarrow L_q$ with $n(\frac1p-\frac1q)=r.$
We will discuss the interpolation properties of the Besov spaces in more detail
in Section \ref{interp} below.
\end{proof}

\subsection{Triebel--Lizorkin spaces}

Similarly to Besov spaces one defines the Triebel--Lizorkin spaces as follows:
\[
F^{r}_{p,q}=F^{r}_{p,q}(G/K)=\left\{f\in\Dcal'(G/K):\,\,\|f\|_{F^{r}_{p,q}}<\infty\right\},
\]
where
\[
 \|f\|_{F^{r}_{p,q}} :=
\Bigg\| \Bigg(\sum_{s=0}^\infty 2^{sr q} \, \Big|\sum_{2^{s}\le\jp{\xi}<2^{s+1}}
 \dxi\ \Tr\p{\whf(\xi)\xi(x)}\Big|^q\Bigg)^{1/q}\Bigg\|_p.
\]

Let us mention several embedding properties of the spaces $F_{p,q}^{r}$.
\begin{thm}\label{THM:Tr-Li}
Let $G/K$ be a compact homogeneous space of dimension $n$. Below, we allow
$r\in\R$ unless stated otherwise.
We have
\begin{itemize}

\item[(1)] \  \   $F_{p,q_1}^{r+\varepsilon}\hookrightarrow F_{p,q_1}^{r}\hookrightarrow F_{p,q_2}^{r}\hookrightarrow F_{p,\infty}^{r},
\qquad 0<\varepsilon,
\quad 0<p\leq\infty,\; 0<q_1\le q_2\le \infty;$

\item[(2)] \ \
$F_{p,q_1}^{r+\varepsilon}\hookrightarrow F_{p,q_2}^{r}, \qquad 0<\varepsilon, \quad
0<p\leq\infty,\; 1\leq q_2<q_1<\infty;$

\item[(3)] \  \
  $F_{p,p}^{r}=B_{p,p}^{r},
  \qquad
0<p<\infty;
  $

\item[(4)] \ \
$B_{p,\min\left\{p,q\right\}}^{r}\hookrightarrow F_{p,q}^{r}\hookrightarrow B_{p,\max\left\{p,q\right\}}^{r},
\qquad
0<p<\infty,\; 0<q\le\infty$.
\end{itemize}
\end{thm}
The proof of this Theorem is similar to the proof of Theorem \ref{THM:Besov}, see also \cite{tribel}.

\section{Wiener and $\beta$-Wiener spaces}
\label{SEC:Wiener}

Let us recall the definition of the Wiener algebra $A$ of absolutely convergent Fourier series
on the circle $\T^{1}$ (see, e.g., Kahane's book \cite{kahane}):
\[
A(\T^{1})=\left\{f:\,\,\|f\|_{A(\T^{1})} =\sum_{j=-\infty}^\infty |\widehat{f}(j)|<\infty\right\}.
\]
With the norm $\ell^{1}(\Gh_{0})$ in \eqref{EQ:Lp-sigmaGK}
this corresponds to
$$
A(G/K)=\left\{f\in \Dcal'(G/K):\,\,
\|f\|_{A(G/K)}:=\|\widehat{f}\|_{\ell^{1}(\Gh_{0})}=\sumxi d_{\xi} k_{\xi}^{1/2} \whfhs<\infty
\right\}.
$$
We will abbreviate this norm to $\|\cdot\|_{A}$.
The main problem here is to determine which smoothness of $f$ guarantees
the absolute convergence of the Fourier series.

It is known that in the case of a unitary group $G$,
 if $f\in C^{k}(G)$ with an even $k>\frac{\dim G}{2}$, then $\whf\in\ell^{1}(\Gh)$ and
hence $f\in A(G)$, see, e.g., Faraut \cite{Faraut}
(in the classical case on the torus this is even more well-known, see, e.g.,
\cite[Th. 3.2.16, p. 184]{Grafakos:Classical-Fourier-Analysis-2008}).

On the other hand, on general compact Lie groups, applying powers of the Laplacian
to the Fourier series, it is also easy to show
that for $s>\frac{\dim G}{2}$, we have $H^s(G)\hookrightarrow A(G)$.
The following theorem sharpens this to the Besov space
$B_{2,1}^{\dim G/2}$, since we can observe that
by Theorem \ref{THM:Besov}, Part (4), we have the embedding
$$
\|f\|_{B_{2,1}^{\dim G/{2}}}\lesssim\|f\|_{H^{s}},\qquad s>\frac12 {\dim G}.
$$
Thus, we sharpen the above embeddings,
also extending it to compact homogeneous spaces.

\begin{thm}\label{wie}
Let $G/K$ be a compact homogeneous space of dimension $n$. Then
\[
\|f \|_{A} \lesssim  \|f\|_{B^{n/2}_{2,1 }}.
\]
\end{thm}
\begin{proof}
We write
\[
\|f\|_{A} =
\sum_{s=0}^\infty F_s, \qquad\mbox{where}\qquad F_s=
\sumxip d_{\xi} k_{\xi}^{1/2} \whfhs .
\]
By H\"{o}lder's inequality, \eqref{EQ:NL1},
and the Plancherel identity \eqref{EQ:Plancherel-partial},
\begin{eqnarray*}
F_s &\le&
\p{\sumxip d_{\xi}k_{\xi}}^{1/2}\p{\sumxip \dxi \whfhs^{2}}^{1/2}
\\
&\le&
2^{sn/2}\left\|
\sumxip \dxi\ \Tr\p{\whf(\xi)\xi(x)}
\right\|_2,
\end{eqnarray*}
and the result follows.
\end{proof}

Let us now study the $\beta$-absolute convergence of the Fourier series, which on the
circle would be
\[
A^\beta(\T^{1})=\left\{f:\,\,\|f\|_{A^\beta} =
\p{\sum_{j=-\infty}^\infty |\widehat{f}(j)|^\beta}^{1/\beta}<\infty\right\}.
\]
This means $\|f\|_{A^\beta}=\|\whf\|_{\ell^{\beta}},$ so that its analogue on
the homogeneous spaces for the family of $\ell^{p}$-norms
\eqref{EQ:norm} becomes
\begin{multline*}
A^\beta(G/K)= \\\left\{f\in\Dcal'(G/K):\,\,
\|f\|_{A^\beta}:=\|\whf\|_{\ell^{\beta}(\Gh_{0})}=
\p{\sum_{[\xi]\in\Gh_{0}} d_{\xi} k_{\xi}^{\beta(\frac1\beta-\frac12)}\whfhs^{\beta}}^{1/\beta}
<\infty\right\},
\end{multline*}
where we can allow any $0<\beta<\infty$. We now analyse its embedding properties.
\begin{thm}\label{THM:Wiener-beta}
Let $G/K$ be a compact homogeneous space of dimension $n$ and let
$1< p\le 2$. Then
\[
\|f \|_{A^\beta} \lesssim  \|f\| _{B^{\alpha n}_{p,\beta}}
\]
for any $\alpha>0$ and $\beta=(\alpha+\frac{1}{p'})^{-1}$.
\end{thm}
\begin{remark}
{\rm (i)} \ In the classical case of functions on the torus, this result was proved by Szasz,
see, e.g., \cite[p.119]{Peetre:bk-New-thoughts-Besov-1976}.

{\rm (ii)} \ Note that the strongest result is when $p=2$,  that is,
\[
\|f \|_{A^\beta} \lesssim  \|f\| _{B^{\alpha^* n}_{2,\beta}} \lesssim  \|f\| _{B^{\alpha n}_{p,\beta}},
\]
where
$\alpha^*, \alpha>0$ and $\beta=(\alpha^*+\frac{1}{2})^{-1}=(\alpha+\frac{1}{p'})^{-1}$.
This follows from Theorem \ref{THM:Besov}, Part (5).
\end{remark}

\begin{proof}[Proof of Theorem \ref{THM:Wiener-beta}]
We can assume that $\widehat{f}(\xi)\not=0$ only for sufficiently large $\jp{\xi}$. We write
\[
\|f\|_{A^\beta}^{\beta} =
\sum_{s=0}^\infty F_s, \qquad\mbox{where}\qquad F_s=
\sumxip d_{\xi} k_{\xi}^{\beta(\frac1\beta-\frac12)}\whfhs^{\beta}.
\]
Let us first assume that $\beta \equiv (\alpha+\frac{1}{p'})^{-1}\ge 2$.
Since $\beta\ge 2$, applying the Hausdorff-Young inequality \eqref{EQ:HY}, we get
\[
F_s
= \sumxip d_{\xi} k_{\xi}^{\beta(\frac1\beta-\frac12)}\whfhs^{\beta}
\le
\left\|
\sumxip \dxi\ \Tr (\whf(\xi)\xi(x))
\right\|_{\beta'}^{\beta}.
\]
Now by Nikolskii's inequality from Theorem \ref{THM2}
for $(L^{\beta'}, L^{p})$ with $\beta<p'$, or equivalently, $p<\beta'$, we have
\begin{eqnarray*}
F_s
&\le& 2^{sn\beta(\frac{1}{p}-\frac{1}{\beta'})}
\left\|
\sumxip \dxi\ \Tr (\whf(\xi)\xi(x))\right\|_{p}^{\beta}\\
&=&
2^{sn\beta(\frac{1}{\beta}-\frac{1}{p'})}
\left\|
\sumxip \dxi\ \Tr (\whf(\xi)\xi(x))
\right\|_{p}^{\beta},
\end{eqnarray*}
which is the required result.

To prove Theorem \ref{THM:Wiener-beta}
in the case of $\beta \equiv (\alpha+\frac{1}{p'})^{-1}< 2$, we put $$\gamma:=\frac{p'}{p'-\beta}>1.$$
By H\"{o}lder's inequality, taking $\delta:=\frac{\beta}{p'}-\frac{\beta}{2}$, so that
$\frac1\gamma+\delta=1-\frac{\beta}{2}=\beta(\frac1\beta-\frac12)$, using \eqref{EQ:NL1}, we get
\begin{eqnarray}\label{EQ:auxFs}
F_s
& = & \sumxip  d_{\xi} k_{\xi}^{\beta(\frac1\beta-\frac12)}\whfhs^{\beta}
\nonumber
\\
&\le&
\p{\sumxip (d_{\xi}^{\frac1\gamma}k_{\xi}^{\frac1\gamma})^{\gamma}}^{1/\gamma}
\p{\sumxip d_{\xi}^{\frac{1}{\gamma'}\gamma'}
k_{\xi}^{\delta\gamma'} \whfhs^{\beta\gamma'}}^{1/\gamma'}
\nonumber
\\
&\lesssim&
2^{sn/\gamma}
\p{\sumxip d_{\xi}
k_{\xi}^{(\frac{\beta}{p'}-\frac{\beta}{2})\frac{p'}{\beta}}
\whfhs^{\beta\gamma'}}^{1/\gamma'}
\nonumber
\\
&=&
2^{\frac{sn(p'-\beta)}{p'}}
\p{\sumxip d_{\xi} k_{\xi}^{p'(\frac{1}{p'}-\frac12)} \whfhs^{p'}}^{\beta/p'}.
\end{eqnarray}

Since $p'\ge 2$, by the Hausdorff-Young inequality \eqref{EQ:HY}, we have
\[
F_s
\le
2^{\frac{sn(p'-\beta)}{p'}}
\left\|
\sumxip \dxi\ \Tr (\whf(\xi)\xi(x))
\right\|_p^{\beta},
\]
i.e.,
\[
\|f\|_{A^\beta} \le
\left[
\sum_{s=0}^\infty
\left(
2^{sn(\frac{1}{\beta}-\frac{1}{p'})}
\left\|
\sumxip \dxi\ \Tr (\whf(\xi)\xi(x))
\right\|_p
\right)^{\beta}
\right]^{1/\beta}
=
\|f\|_{B^{n\alpha}_{p,\beta}},
\]
completing the proof.
\end{proof}

The converse to Theorem \ref{THM:Wiener-beta} is as follows:
\begin{thm}[Inverse result] \label{THM:Wiener-inverse}
Let $2\le p <\infty$, then
\[
 \|f\| _{B^{\overline{\alpha}n}_{p,\beta}}  \lesssim \|f \|_{A^\beta}
\]
for  $\overline{\alpha}:=\min\{\alpha, 0\}$, $\beta=(\alpha+\frac{1}{p'})^{-1}>0$.
\end{thm}

\begin{remark}
{\rm (i)} \ If $\beta\ge 2$, then the strongest result is when $p=2$,  that is,
\[
 \|f\| _{B^{{\alpha}n}_{p,\beta}}  \lesssim
 \|f\| _{B^{{\alpha^*}n}_{2,\beta}}  \lesssim \|f \|_{A^\beta},
\]
where $\beta=(\alpha^*+\frac{1}{2})^{-1}=(\alpha+\frac{1}{p'})^{-1}$.
This follows from Theorem \ref{THM:Besov}, Part (5).
Note that in this case
$\alpha^*, \alpha<0$, i.e., $\overline{\alpha}=\alpha$ and $\overline{\alpha^*}=\alpha^*$.

{\rm (ii)} \
If $\beta< 2$, then the strongest result is when $p=\beta'$,  that is,
\[
 \|f\| _{B^{{\alpha}n}_{p,\beta}}  \lesssim
 \|f\| _{B^{0}_{\beta',\beta}}  \lesssim \|f \|_{A^\beta}.
\]
This follows from Theorem \ref{THM:Besov}, Part (5).
\end{remark}

\begin{proof}[Proof of Theorem \ref{THM:Wiener-inverse}]
Let first $\beta\le p'$, or equivalently, $\alpha \ge 0$. In this case $\overline{\alpha}=0$.
 Since $p\ge 2$, then applying the Hausdorff-Young inequality \eqref{EQ:HY}, we get
\begin{eqnarray*}
 \|f\| _{B^{0}_{p,\beta}}  & = &
\left[
\sum_{s=0}^\infty
\left\|
\sumxip   \dxi\ \Tr (\whf(\xi)\xi(x))
\right\|_p^\beta
\right]^{1/\beta}
\\&\le&
\left[
\sum_{s=0}^\infty
\left(
\sumxip   d_{\xi} k_{\xi}^{p'(\frac{1}{p'}-\frac12)}\whfhs^{p'}
\right)^{\beta/p'}
\right]^{1/\beta}
\\&\lesssim&
\left[
 \sum_{s=0}^\infty \sumxip   d_{\xi} k_{\xi}^{\beta(\frac1\beta-\frac12)}\whfhs^{\beta}
 \right]^{1/\beta}
=
\|f\|_{A^\beta},
\end{eqnarray*}
where in the last line we have used
 the inequality $\|\whf\|_{\ell^{p'}(\Gh_{0})}\le \|\whf\|_{\ell^\beta(\Gh_{0})}$ for
$\beta\leq p'$, see \eqref{EQ:cpt-lp-lq}.

Let now $\beta> p'$, or equivalently, $\beta'< p$, i.e., $\overline{\alpha}=\alpha<0$.
First, we observe that by the H\"{o}lder inequality, we have
\begin{eqnarray*}
& &\sumxip d_{\xi} k_{\xi}^{p'(\frac{1}{p'}-\frac12)}\whfhs^{p'}
\\
&\leq &
\p{\sumxip \p{((d_\xi k_{\xi})^{\frac{\beta-p'}{\beta}}}^{\frac{\beta}{\beta-p'}}}^
{1-\frac{p'}{\beta}}
\p{\sumxip
\p{d_{\xi}^{\frac{p'}{\beta}} k_\xi^{p'(\frac{1}{\beta}-\frac12)} \whfhs^{p'}}^{\frac{\beta}{p'}}}^
{\frac{p'}{\beta}}
\\
& \lesssim &
2^{sn(1-\frac{p'}{\beta})}
\p{\sumxip
d_{\xi} k_\xi^{\beta(\frac{1}{\beta}-\frac12)} \whfhs^{\beta}}^{\frac{p'}{\beta}}.
\end{eqnarray*}
Using this and $p\geq 2$, by the Hausdorff-Young inequality
\eqref{EQ:HY}, we get
\begin{eqnarray*}
 \|f\| _{B^{\alpha n}_{p,\beta}}  &\le &
\left[
\sum_{s=0}^\infty
2^{s\alpha n\beta}
\left(
\sumxip d_{\xi} k_{\xi}^{p'(\frac{1}{p'}-\frac12)}\whfhs^{p'}
\right)^{\beta/p'}
\right]^{1/\beta}
\\
&\le&
\left[
\sum_{s=0}^\infty
2^{s\alpha n \beta}
\left(
2^{s n \big(\frac{\beta}{p'}-1\big)}
\sumxip  d_{\xi} k_{\xi}^{\beta(\frac{1}{\beta}-\frac12)}\whfhs^{\beta}
\right)
\right]^{1/\beta}
\\
&=&
\left[
 \sum_{s=0}^\infty \sumxip d_{\xi} k_{\xi}^{\beta(\frac{1}{\beta}-\frac12)}\whfhs^{\beta}
  \right]^{1/\beta}
=
\|f\|_{A^\beta}.
\end{eqnarray*}
where in the second line we used \eqref{EQ:auxFs} with $\beta$ and $p'$
interchanged.
\end{proof}

\section{Beurling and $\beta$-Beurling spaces}
\label{SEC:Beurling}

Let us recall the definition of the Beurling space on the circe $\T^{1}$:
\begin{equation}\label{DEF:Beurling}
A^*(\T^{1})=\left\{f:\,\,\|f\|_{A^*(\T^{1})} =\sum_{j=0}^\infty \sup\limits_{j\le |k|}|\widehat{f}(k)|
<\infty\right\}.
\end{equation}

The space $A^*$ was introduced by Beurling \cite{Beurling:AM-synthesis}
for establishing contraction
properties of functions.
In \cite{bel} it was shown that $A^*(\T^{1})$ is an algebra and its properties were investigated.
We note that $\|f\|_{A^*}$ can be represented as follows:
\begin{equation}\label{DEF:Beurling2}
\|f\|_{A^*(\T^{1})} \asymp \sum_{s=0}^\infty 2^s \sup\limits_{2^s\le |k|<2^{s+1}} |\widehat{f}(k)|.
\end{equation}
Indeed, we have
\begin{eqnarray*}
\|f\|_{A^*(\T^{1})} &=&
\sum_{j=0}^\infty \sup\limits_{j\le |k|}|\widehat{f}(k)|
\asymp
\sum_{s=0}^\infty \sum_{2^s\le j<2^{s+1}} \sup\limits_{j\le |k|}|\widehat{f}(k)|
\asymp
\sum_{s=0}^\infty 2^s\sup\limits_{2^s\le |k|}|\widehat{f}(k)|
\\&\asymp&
\sum_{s=0}^\infty 2^s\sup\limits_{s\le l } \sup\limits_{2^l\le|k|<2^{l+1}}|\widehat{f}(k)|=:J.
\end{eqnarray*}
It is clear that
$J\ge
\sum_{s=0}^\infty 2^s \sup\limits_{2^s\le|k|<2^{s+1}}|\widehat{f}(k)|.
$
On the other hand,
\begin{eqnarray*}
J&\le&
\sum_{s=0}^\infty 2^s \sum_{l=s}^\infty \sup\limits_{2^l\le|k|<2^{l+1}}|\widehat{f}(k)|
=
\sum_{l=0}^\infty  \sup\limits_{2^l\le|k|<2^{l+1}}|\widehat{f}(k)|\sum_{s=0}^l 2^s
\\
&\asymp&\sum_{l=0}^\infty 2^l \sup\limits_{2^l\le|k|<2^{l+1}}|\widehat{f}(k)|.
\end{eqnarray*}

In our case the space $A^{*}(G/K)$ on a compact homogeneous space
$G/K$ of dimension $n$,
analogous to that in \eqref{DEF:Beurling2} on the circle,
for the $\ell^{\infty}$-norm \eqref{EQ:norm-linfty},
becomes
\begin{multline}\label{representation}
A^*(G/K)= \\
\left\{f\in\Dcal'(G/K):\,\
\|f\|_{A^*(G/K)} :=  \sum_{s=0}^\infty 2^{ns} \sup\limits_{2^s\le \jp{\xi}<2^{s+1}}
k_{\xi}^{-1/2}\whfhs<\infty\right\}.
\end{multline}

In fact, we can analyse a more general scale of spaces, the $\beta$-version of
these spaces.
Such function spaces play an important role in the summability theory and
in the Fourier synthesis (see, e.g.,
\cite[Theorems 1.25 and 1.16]{stein-weiss} and \cite[Theorem 8.1.3, Ch. 6]{trigub}).

Thus, for any $0<\beta<\infty$,
we define $A^{*,\beta}$ by
\begin{equation}\label{EQ:beta-Beurling-norm2}
\|f\|_{A^{*,\beta}} := \left(\sum_{s=0}^\infty 2^{ns} \left(\sup\limits_{2^s\le \jp{\xi}}
k_{\xi}^{-1/2}\whfhs\right)^\beta\right)^{1/\beta}<\infty.
\end{equation}
We note that for convenience we change the range of the supremum in $\xi$
from a dyadic strip in \eqref{representation} to an infinite set
in \eqref{EQ:beta-Beurling-norm2}, but we can show that this change
produces equivalent norms.
Thus, we show that we have $A^{*}=A^{*,1}$ and, more generally,
for any $0<\beta<\infty$ we have the
equivalence
\begin{equation}\label{EQ:beta-Beurling-norms}
\sum_{s=0}^\infty 2^{ns} \left(\sup\limits_{2^s\le \jp{\xi}< 2^{s+1}}
k_{\xi}^{-1/2}\whfhs\right)^\beta \asymp
\sum_{s=0}^\infty 2^{ns} \left(\sup\limits_{2^s\le \jp{\xi}}
k_{\xi}^{-1/2}\whfhs\right)^\beta.
\end{equation}
Indeed, the inequality $\leq$ is trivial. Conversely, we write
\begin{eqnarray*}
\sum_{s=0}^\infty 2^{ns} \left(\sup\limits_{2^s\le \jp{\xi}}
k_{\xi}^{-1/2}\whfhs\right)^\beta
& \asymp &
\sum_{s=0}^\infty 2^{ns} \sup_{s\leq l}\sup\limits_{2^l\le \jp{\xi}< 2^{l+1}}
k_{\xi}^{-\beta/2}\whfhs^\beta \\
& \leq &
\sum_{s=0}^\infty 2^{ns} \sum_{l=s}^{\infty}\sup\limits_{2^l\le \jp{\xi}< 2^{l+1}}
k_{\xi}^{-\beta/2}\whfhs^\beta \\
& = &
\sum_{l=0}^\infty \sup\limits_{2^l\le \jp{\xi}< 2^{l+1}}
k_{\xi}^{-\beta/2}\whfhs^\beta  \sum_{s=0}^{l} 2^{ns} \\
& \asymp&
\sum_{l=0}^\infty 2^{l s} \sup\limits_{2^l\le \jp{\xi}<2^{l+1}}
k_{\xi}^{-\beta/2}\whfhs^\beta,
\end{eqnarray*}
proving \eqref{EQ:beta-Beurling-norms}. So, we can work with either of the
equivalent expressions in \eqref{EQ:beta-Beurling-norms}.

We now prove the following embedding properties between Beurling's and Besov spaces:
\begin{thm}\label{beur}
Let $G/K$ be a compact homogeneous space of dimension $n$.
Let $0<\beta<\infty$ and $p\geq 2$.
Then we have
$$
 \|f\|_{B^{n(\frac1\beta-\frac{1}{p'})}_{p,\beta}}
  \lesssim
\|f\|_{A^{*,\beta}}\lesssim  \|f\|_{B^{n/\beta}_{1,\beta}}.
$$
\end{thm}

\begin{proof}
Again, we may assume
that $\whf(\xi)=0$ for small $\jp{\xi}$.
To prove the left-hand side inequality, using \eqref{EQ:HY} and $p\geq 2$, we get
\begin{eqnarray*}
&&  \|f\|_{B^{n(1/\beta-1/p')}_{p,\beta}} \\
&  =&
\p{\sum_{s=0}^\infty
2^{sn(1-\beta/p') }
\left\|
\sumxip \dxi\ \Tr (\whf(\xi)\xi(x))
\right\|_p^{\beta}}^{1/\beta}
\\
&\leq&
\p{\sum_{s=0}^\infty
2^{sn(1-\beta/p')}
\left(
\sumxip \dxi k_{\xi}^{p'(\frac{1}{p'}-\frac12)}\whfhs^{p'}
\right)^{\beta/p'}}^{1/\beta}
\\
&\le&
\p{\sum_{s=0}^\infty
2^{sn(1-\beta/p')}
\left(
\sumxip d_{\xi} k_{\xi}
\right)^{\beta/p'}
\p{\sup\limits_{2^s\le \jp{\xi}<2^{s+1}} k_{\xi}^{-1/2} \whfhs}^{\beta}}^{1/\beta}
\\
&\le&
\p{\sum_{s=0}^\infty
2^{sn}
\p{\sup\limits_{2^s\le \jp{\xi}<2^{s+1}} k_{\xi}^{-1/2} \whfhs}^{\beta}}^{1/\beta}
=\|f\|_{A^{*,\beta}}.
\end{eqnarray*}
To show the right-hand side inequality,
we denote
$$
S_{l}f(x):=\sum_{\jp{\eta}\leq l-1} d_{\eta}\ \Tr(\whf(\eta)\eta(x)).
$$
Then, by the orthogonality of representation coefficients, we have
$
\widehat{S_{l}f}(\xi)=0
$
for any $[\xi]\in\Gh_{0}$ with $\jp{\xi}\geq l.$
Consequently, we note that we have
$\whf(\xi)=\widehat{(f-S_{2^{s}}f)}(\xi)$ if $\jp{\xi}\geq 2^{s}.$
Using these observations and the Hausdorff-Young inequality, we can estimate
\begin{eqnarray*}
\sup\limits_{2^s\le \jp{\xi}<2^{s+1}} k_{\xi}^{-1/2} \whfhs
&=&
\sup\limits_{2^s\le \jp{\xi}<2^{s+1}} k_{\xi}^{-1/2} \|\whf(\xi)-\widehat{S_{2^{s}}f}(\xi)\|_{\HS}
\\&\leq&
\|\widehat{f-S_{2^{s}}f}\|_{\ell^{\infty}(\Gh_{0})}
\leq
\|f-S_{2^{s}}f\|_{L^{1}(G/K)}.
\end{eqnarray*}
Then
\begin{eqnarray*}
\sup\limits_{2^s\le \jp{\xi}<2^{s+1}} k_{\xi}^{-1/2} \whfhs
&\le&
 \|f-S_{2^s}(f)\|_1 \le \sum_{k=s}^\infty \|S_{2^{k+1}}(f)-S_{2^k}(f)\|_1.
\end{eqnarray*}
Therefore,
\begin{eqnarray*}
\|f\|_{A^{*,\beta}} &\le& C
\p{\sum_{s=0}^\infty 2^{sn} \p{\sup\limits_{2^s\le \jp{\xi}<2^{s+1}} k_{\xi}^{-1/2} \whfhs}^{\beta}
}^{1/\beta}
\\
&\le& C \p{\sum_{s=0}^\infty 2^{sn}
\Big\|f-S_{2^s}(f)\Big\|_1
^{\beta}
}^{1/\beta}.
\end{eqnarray*}
Using Proposition \ref{proposit} we get
$$
\|f\|_{A^{*,\beta}}
\lesssim
\|f\|_{B^{n/\beta}_{1,\beta}}
$$
completing the proof.
\end{proof}

Theorems \ref{wie} and  \ref{beur} imply, taking $p=2$:
\begin{cor}
For $0<\beta<\infty$, we have
$$
\|f\|_{A^{\beta}}\lesssim
 \|f\|_{B^{n(1/\beta-1/2)}_{2,\beta}}
  \lesssim
\|f\|_{A^{*,\beta}}.
$$
In particular, taking $\beta=1$, we have
$$
\|f\|_{A}\lesssim
 \|f\|_{B^{n/2}_{2,1}}
  \lesssim
\|f\|_{A^*}.
$$
\end{cor}

\section{Interpolation}
\label{interp}

Let $X_0, X_1$ be two Banach spaces, with $X_1$ continuously embedded in $X_0: X_1\hookrightarrow X_0$. We define the $K$-functional for $f\in X_0+X_1$ by
$$
 K(f, t; X_0, X_1)
:= \inf_{f=f_0+f_1} (\|f_0\|_{X_0} + t\|f_1\|_{X_1}),\;\;t\ge 0.
$$
To investigate intermediate space $X$ for the pair $(X_0, X_1)$, i.e.,
$X_0\cap X_1 \subset X \subset X_0+ X_1$, one uses the $\theta, q$-interpolation spaces.

For $0< q < \infty$,  $0<\theta<1$, we define
$$
(X_0,X_1)_{\theta,q} := \left\{ f \in X_0 + X_1 :
\|f\|_{(X_0,X_1)_{\theta ,q}}
= \left (\int\limits_{0}^{\infty} (t^{-\theta}K(f,t))^q\frac  {{dt}}{{t}}
\right )^{\frac {1}{q}} < \infty  \right\} ,
$$
and for $q=\infty$,
$$(X_0,X_1)_{\theta,\infty} := \left\{ f \in X_0 + X_1 :
\|f\|_{(X_0,X_1)_{\theta ,\infty}}
= \sup_{0<t<\infty} t^{-\theta}K(f,t) < \infty \right\} .
$$
For technical convenience, we define the following Beurling-type spaces:

$$
A_r^{*,\beta}=\left\{f: \|f\|_{A_r^{*,\beta}}
:=\left(\sum_{s=0}^\infty\Big(2^{r ns}\sup\limits_{2^s\le \jp{\xi}}
k_{\xi}^{-1/2}\|\widehat{f}(\xi)\|_{\HS}\Big)^\beta\right)^{\frac1\beta}<\infty\right\}.
$$
Note that $A_{1/\beta}^{*,\beta}=A^{*,\beta}$, where $A^{*,\beta}$ is given by
the norm (\ref{EQ:beta-Beurling-norm2}).

\begin{thm} \label{THM:interpolation}
Let $0<r_1 < r_0 < \infty$, $0< \beta_0, \beta_1, q \le \infty$,  and  $$ r=(1-\theta)r_0+\theta r_1.$$
\begin{itemize}
\item[(i)] \  \ We have
$$
(
A_{r_0}^{*,\beta_0},
A_{r_1}^{*,\beta_1}
)_{\theta, q} = A_r^{*,q}.
$$
In particular,
$$
(A^{*,1/r_0},A^{*,1/r_1})_{\theta, 1/r} = A^{*,1/r}.
$$
\item[(ii)] \  \
 If $0<p \le \infty$,
$$
(B_{p,\beta_0}^{r_0},B_{p,\beta_1}^{r_1})_{\theta ,q} = B_{p,q}^r.
$$
\item[(iii)] \  \
If $1< p < \infty$,
$$
(H_{p}^{r_0},H_{p}^{r_1})_{\theta, q} = B_{p,q}^{r}.
$$

\item[(iv)] \  \
If $0<p< \infty$ ,
$$
(F_{p,\beta_0}^{r_0},F_{p, \beta_1}^{r_1})_{\theta ,q} = B_{p,q}^r.
$$
\end{itemize}
\end{thm}

\begin{proof}
Let $f \in (A_{r_0}^{*,\beta_0}, A_{r_1}^{*,\beta_1})_{\theta, q}$. Take any representation $f=f_0+f_1$ such that  $f_0 \in
 A_{r_0}^{*,\beta_0}$ and $f_1 \in A_{r_1}^{*,\beta_1}.$ Then for any $s\in \Bbb Z_+$ 
 we get
\begin{eqnarray*}
2^{r ns}\sup\limits_{2^s\le \jp{\xi}}
k_{\xi}^{-1/2}\|\widehat{f}(\xi)\|_{\HS}
 &\le& 2^{(r-r_0)ns}\Big(2^{r_0 ns}\sup\limits_{2^s\le \jp{\xi}}
k_{\xi}^{-1/2}\|\widehat{f_0}(\xi)\|_{\HS})
\\&+&
2^{(r_0-r_1)n s} 2^{r_1 ns}\sup\limits_{2^s\le \jp{\xi}}
k_{\xi}^{-1/2}\|\widehat{f_1}(\xi)\|_{\HS}\Big)
\\
&\lesssim&
 2^{(r-r_0)n s}\Bigg(\Big(\sum_{r=0}^s 2^{r_0 ns\beta_0}\Big)^{1/\beta_0}\sup\limits_{2^s\le \jp{\xi}}
k_{\xi}^{-1/2}\|\widehat{f_0}(\xi)\|_{\HS}
\\
&+&
2^{(r_0-r_1)n s} \Big(\sum_{r=0}^s 2^{r_1 n r\beta_1}\Big)^{1/{\beta_1}}\sup\limits_{2^s\le \jp{\xi}}
k_{\xi}^{-1/2}\|\widehat{f_1}(\xi)\|_{\HS}\Bigg)
\\
&\lesssim&
2^{(r-r_0)n s}\Big(\|f_0\|_{A_{r_0}^{*,\beta_0}}  + 2^{(r_0-r_1)n s}
\|f_1\|_{A_{r_1}^{*,\beta_1}}\Big).
\end{eqnarray*}
Taking into account our choice of $f_0$ and $f_1$, by the definition of the $K$-functional
$K(f,t):=K(f, t; A_{r_0}^{*,\beta_0}, A_{r_1}^{*,\beta_1}),$
we have
$$
2^{r ns}\sup\limits_{2^s\le \jp{\xi}}
k_{\xi}^{-1/2}\|\widehat{f}(\xi)\|_{\HS}\lesssim 2^{(r-r_0) ns} K(f,2^{(r_0-r_1)n s})= 2^{-\theta(r_0-r_1) ns} K(f,2^{(r_0-r_1)n s}).
$$
Therefore,
\begin{eqnarray*}
\|f\|_{A_r^{*,q}}&\lesssim&\left(\sum_{s=0}^\infty\left( 2^{-\theta(r_0-r_1) n s}K(f,2^{(r_0-r_1)n s})\right)^q\right)^{\frac1q}
\\
&\asymp& \left(\int_{1}^\infty\left( t^{-\theta(r_0-r_1) n }K(f,t^{(r_0-r_1)n })\right)^q\frac{dt}{t}\right)^{\frac1q}
\\&\lesssim&
\left(\int_{0}^\infty( t^{-\theta}K(f,t))^q\frac{dt}{t}\right)^{\frac1q}= \|f\|_{(
A_{r_0}^{*,\beta_0}, A_{r_1}^{*,\beta_1} )_{\theta, q}},
\end{eqnarray*}
i.e.,
$$
(
A_{r_0}^{*,\beta_0},
A_{r_1}^{*,\beta_1}
)_{\theta, q} \hookrightarrow A_r^{*,q}.
$$

Let us show the inverse embedding. Let $f \in A_r^{*,q}$, $\tau = \min(\beta_0, \beta_1, q),$ and $r \in \Bbb Z_+$.
 Define $f_0$ and $f_1$ as follows:
$$
f_0(x):=S_{2^l}f(x)=  \sum_{\jp{\xi}\leq 2^l} d_{\xi} \Tr (\widehat{f}(\xi) \xi(x)),
$$
and
$$
f_1 := f - f_0.
$$
Then
\begin{align*}
&\|f_0\|_{
A_{r_0}^{*,\beta_0}}=
\left(\sum_{s=0}^\infty\Big(2^{r_0 ns}\sup\limits_{2^s\le \jp{\xi}}
k_{\xi}^{-1/2}\|\widehat{f_0}(\xi)\|_{\HS}\Big)^{\beta_0}\right)^{\frac1{\beta_0}}\\
&\le \left(\sum_{s=0}^\infty\Big(2^{r_0 ns}\sup\limits_{2^s\le \jp{\xi}}
k_{\xi}^{-1/2}\|\widehat{f_0}(\xi)\|_{\HS}\Big)^{\tau}\right)^{\frac1{\tau}}=  \left(\sum_{s=0}^l\Big(2^{r_0 ns}\sup\limits_{2^s\le \jp{\xi}}
k_{\xi}^{-1/2}\|\widehat{f}(\xi)\|_{\HS}\Big)^{\tau}\right)^{\frac1{\tau}}
\end{align*}
and
\begin{align*}
&\|f_1\|_{
A_{r_1}^{*,\beta_1}}
\lesssim
2^{r_1 nl}\sup\limits_{2^l\le \jp{\xi}}
k_{\xi}^{-1/2}\|\widehat{f}(\xi)\|_{\HS} +\left(\sum_{s=l+1}^\infty\Big(2^{r_1 ns}\sup\limits_{2^s\le \jp{\xi}}
k_{\xi}^{-1/2}\|\widehat{f}(\xi)\|_{\HS}\Big)^{\tau}\right)^{\frac1{\tau}}
\\
&\lesssim 
2^{(r_1-r_0) nl}\left(\sum_{s=0}^l\Big(2^{r_0 ns}\sup\limits_{2^s\le \jp{\xi}}
k_{\xi}^{-1/2}\|\widehat{f}(\xi)\|_{\HS}\Big)^{\tau}\right)^{\frac1{\tau}} + \left(\sum_{s=l+1}^\infty\Big(2^{r_1 ns}\sup\limits_{2^s\le \jp{\xi}}
k_{\xi}^{-1/2}\|\widehat{f}(\xi)\|_{\HS}\Big)^{\tau}\right)^{\frac1{\tau}}.
\end{align*}
Consider now
\begin{eqnarray*}
\|f\|_{(
A_{r_0}^{*,\beta_0},
A_{r_1}^{*,\beta_1}
)_{\theta, q}} &=& \left(\int_{0}^\infty( t^{-\theta}K(f,t))^q\frac{dt}{t}\right)^{\frac1q}
\\&=& \frac1{(r_0 - r_1)^{1/q}}(\left(\int_{0}^\infty\left( t^{-\theta(r_0-r_1) n }K(f,t^{(r_0-r_1)n })\right)^q\frac{dt}{t}\right)^{\frac1q}.
\end{eqnarray*}
Since
$
K(f,t^{(r_0-r_1)n})=\inf_{f=f_0+f_1}(\|f_0\|_{
A_{r_0}^{*,\beta_0}}+t^{(r_0-r_1)n}\|f_1\|_{
A_{r_1}^{*,\beta_1}
})\leq t^{(r_0-r_1)n}\|f\|_{
A_{r_1}^{*,\beta_1}
},
$
we have
\begin{eqnarray*}
\|f\|_{(
A_{r_0}^{*,\beta_0},
A_{r_1}^{*,\beta_1}
)_{\theta, q}} &\lesssim&
\Big[\int_{0}^1\Big( t^{-\theta(r_0-r_1) n }t^{(r_0-r_1)n }  \|f\|_{
A_{r_1}^{*,\beta_1}
}  \Big)^q\frac{dt}{t}
\\&+& \int_{1}^\infty \Big( t^{-\theta(r_0-r_1) n }K(f,t^{(r_0-r_1)n }) \Big)^q\frac{dt}t\Big]^{1/q}
\\&\asymp&
 \|f\|_{
 A_{r_1}^{*,\beta_1}
} + \left(\sum_{l=0}^\infty\left( 2^{-\theta(r_0-r_1) n l}K(f,2^{(r_0-r_1)n l})\right)^q\right)^{\frac1q}.
\end{eqnarray*}
 In view of  $r_1 < r$, we get $\|f\|_{
A_{r_1}^{*,\beta_1}
} \lesssim \|f\|_{
A_{r}^{*,\beta}
}$.
Then, using 
$$K(f,2^{(r_0-r_1)n l})\le
\|f_0\|_{
A_{r_0}^{*,\beta_0}}
+ 2^{(r_0-r_1)n l}
\|f_1\|_{
A_{r_1}^{*,\beta_1}}
$$ and estimates above,
we have
\begin{eqnarray*}
\|f\|_{(
A_{r_0}^{*,\beta_0},
A_{r_1}^{*,\beta_1}
)_{\theta, q}}&\lesssim&
\|f\|_{
A_{r}^{*,\beta}
} + 
\left(\sum_{l=0}^\infty 2^{-\theta(r_0-r_1)q n l}\Bigg\{
\Big[\sum_{s=0}^l (2^{r_0 ns}\sup\limits_{2^s\le \jp{\xi}}
k_{\xi}^{-1/2}\|\widehat{f}(\xi)\|_{\HS})^{\tau}\Big]^{\frac1{\tau}}
\right.
\\&+&
\left.
 2^{(r_0-r_1) n l} \Big[\sum_{s=l+1}^\infty(2^{r_1 ns}\sup\limits_{2^s\le \jp{\xi}}
k_{\xi}^{-1/2}\|\widehat{f}(\xi)\|_{\HS})^{\tau}\Big]^{\frac1{\tau}}\Bigg\}^q\right)^{\frac1q}.
\end{eqnarray*}
Further, since $\tau \le q$ we apply Hardy's inequality to get
$$
\|f\|_{(
A_{r_0}^{*,\beta_0},
A_{r_1}^{*,\beta_1}
)_{\theta, q}} \lesssim \|f\|_{
A_{r}^{*,\beta}
}.
$$
This completes the proof of (i).

To show (ii), we first note that $(B_{p, \beta_0}^{r_0},B_{p,\beta_1}^{r_1})_{\theta, q} \hookrightarrow B_{p,q}^{r}$ can be proved similarly to the proof of (i) using the embedding  $B_{p, \beta}^{r}\hookrightarrow B_{p, \infty}^{r}$ from Theorem \ref{THM:Besov} (1).

Let us verify the inverse embedding. We have
\begin{align*}
&\|f\|_{(B_{p,\beta_0}^{r_0},B_{p,\beta_1}^{r_1})_{\theta, q}}
\lesssim
\|f\|_{B_{p,\beta_1}^{r_1}} + \left(\sum_{k=0}^\infty\left( 2^{-\theta(r_0-r_1)  k}K(f,2^{(r_0-r_1) k};B_{p,\beta_0}^{r_0},B_{p,\beta_1}^{r_1})\right)^q\right)^{\frac1q}
\\&\lesssim
\|f\|_{B_{p,q}^{r}} +
\left(\sum_{k=0}^\infty\left[ 2^{-\theta(r_0-r_1)  k}\left(\|S_{2^k}(f)\|_{B_{p,\beta_0}^{r_0}}+2^{(r_0-r_1) k}\|f-S_{2^k}(f)\|_{B_{p,\beta_1}^{r_1}}\right)\right]^q\right)^{\frac1q}
\\&\lesssim
\|f\|_{B_{p,q}^{r}} + \left(\sum_{k=0}^\infty 2^{-\theta(r_0-r_1)q  k}\left(\|S_{2^k}(f)\|_{B_{p,\tau}^{r_0}}+2^{(r_0-r_1) k}\|f-S_{2^k}(f)\|_{B_{p,\tau}^{r_1}}\right)^q\right)^{\frac1q}
\\&\lesssim
\|f\|_{B_{p,q}^{r}} + \left(\sum_{k=0}^\infty 2^{-\theta(r_0-r_1)q  k}\left\{\Big[\sum_{s=0}^k \big(2^{r_0s}\|S_{2^k}(f)-S_{2^s}(f)\|_{L_p}\big)^\tau\Big]^{1/\tau} \right.\right.
\\
&
\left.\left.
\qquad\qquad\;+\,
2^{(r_0-r_1)k}
\Big[\sum_{s=k+1}^\infty\big(2^{r_1s}\|f-S_{2^s}(f)\|_{B_{p,\tau}^{r_1}}\big)^\tau\Big]^{1/\tau}
\right\}^q\right)^{\frac1q}
\lesssim
 \|f\|_{B_{p,q}^{r}},
\end{align*}
where in the last estimate we have used Hardy's inequality.

Part (iii) follows from
$$
B_{p,1}^{r}\hookrightarrow H_{p}^{r} \hookrightarrow B_{p,\infty}^{r}
$$
and
$$
B_{p,q}^r=(B_{p,1}^{r_0},B_{p,1}^{r_1})_{\theta, q} \hookrightarrow(H_p^{r_0},H_p^{r_1})_{\theta, q} \hookrightarrow(B_{p,\infty}^{r_0},B_{p,\infty}^{r_1})_{\theta, q}=B_{p,q}^r;
$$
see Theorem \ref{THM:Besov} (6)--(7).

Let us finally prove (iv).
Since $B_{p,\min\left\{p,q\right\}}^{r}\hookrightarrow F_{p,q}^{r}\hookrightarrow B_{p,\max\left\{p,q\right\}}^{r},$
we have
\begin{eqnarray*}
B_{p,q}^r=(B_{p,\min\left\{p,\beta_0\right\}}^{r_0},B_{p,\min\left\{p,\beta_1\right\}}^{r_1})_{\theta, q} &\hookrightarrow&(F_{p,\beta_0}^{r_0},F_{p,\beta_1}^{r_1})_{\theta, q} \\&\hookrightarrow&(B_{p,\max\left\{p,\beta_0\right\}}^{r_0},B_{p,\max\left\{p,\beta_1\right\}}^{r_1})_{\theta, q}=B_{p,q}^r.
\end{eqnarray*}
The proof is complete.
\end{proof}
Above we have provided rather direct proofs of the interpolation theorems.
Another proof of such results could be also obtained using other methods, see, e.g.,
\cite{Bergh-Lofstrom:BK-interpolation-spaces-1976,Burenkov-Darbaeva-Nursultanov:2013}.

\section{Localisation of Besov spaces}
\label{SEC:localisations}

In this section we show that as a corollary of Theorem \ref{THM:interpolation}, (iii),
for certain ranges of indices,
the localisations of the Besov spaces \eqref{EQ:Besov1}-\eqref{EQ:Besov2}
coincide with the usual Besov spaces on $\Rn$, so that the norm
\eqref{EQ:Besov2} provides the global characterisation of Besov spaces defined
on the space $G/K$ considered as a smooth manifold. We note that for
Sobolev spaces such characterisation is much simpler and follows directly
from the elliptic regularity, see \eqref{EQ:Sobolev-norms} and the discussion
before it.

For $x,h\in\Rn$, let us denote
$$
\Delta_{h}^{m} f(x):=\sum_{k=0}^{m} C_{m}^{k} (-1)^{m-k} f(x+kh)
$$
and
$$
\omega_{p}^{m}(t,f):=\sup_{|h|\leq t}\|\Delta_{h}^{m}f\|_{L_{p}}.
$$
Then it is known that for $r>0$ and $1\leq p,q\leq\infty$,
the Besov space $B_{p,q}^{r}(\Rn)$ on $\Rn$ can be characterised by the
difference condition
\begin{equation}\label{EQ:Besov-Rn}
\|f\|_{B_{p,q}^{r}(\Rn)}\asymp \|f\|_{L_{p}}+\sum_{j=1}^{n}
\p{\int_{0}^{\infty} \p{t^{-r}\omega_{p}^{m}(t,f)}^{q}\frac{dt}{t}}^{1/q},
\end{equation}
with a natural modification for $q=\infty$, see, e.g.,
\cite[Prop. 1.18]{Wang-et-al:Bk-WS-2011}, or
\cite[(1.13)]{Triebel:Bk-Function-Spaces-III-2006} for a slightly different
expression. From \eqref{EQ:Besov-Rn}  we see that
for $r>0$ and $1\leq p,q\leq\infty$, the Besov spaces
$B_{p,q}^{r}(\Rn)$ are invariant under smooth changes of variables and can,
therefore, we defined on arbitrary smooth manifolds. Consequently,
the same is true for any $r<0$ by duality and, in fact, for any $r\leq 0$ by
the property $(1-\Delta)^{s/2}B_{p,q}^{r}=B_{p,q}^{r-s}.$

We say that such an extension
is the Besov space on the manifold defined by localisations.
For Sobolev spaces $H_{p}^{r}$ we can use the same terminology, and
the equivalence of norms \eqref{EQ:Sobolev-norms} says that
the Sobolev space $H_{p}^{r}$ on $G/K$ defined by localisations coincides with
the Sobolev space defined by the norm on the right hand side of \eqref{EQ:Sobolev-norms}.
We now formulate the analogue of this for Besov spaces.

\begin{thm}\label{COR:Besovs}
Let $G/K$ be the compact homogeneous space and let
$1<p<\infty$, $1\leq q\leq\infty$ and $r\in\R$. Then the Besov
space $B_{p,q}^{r}$ on $G/K$ defined by localisations coincides with
the Besov space $B_{p,q}^{r}(G/K)$ defined by
\eqref{EQ:Besov1}--\eqref{EQ:Besov2}, with the equivalence of norms.
\end{thm}
\begin{proof}
Let first $r>0$.
By Theorem  \ref{THM:interpolation}, (iii),
we know that the Besov space $B_{p,q}^{r}(G/K)$ defined by
\eqref{EQ:Besov1}--\eqref{EQ:Besov2} is the interpolation space for
Sobolev spaces with the norms given on the right hand side of
\eqref{EQ:Sobolev-norms}.
Since Sobolev spaces with such norms coincide with their localisation,
the statement of Theorem \ref{COR:Besovs} follows from the
corresponding interpolation property of Sobolev spaces on $\Rn$,
see \cite{tribel}.
The statement for $r<0$ follows by duality or, in fact for any $r\leq 0$,
since the property $(1-\Delta)^{s/2}B_{p,q}^{r}=B_{p,q}^{r-s}$
holds for an elliptic operator $\Delta$ for both scales of spaces.
\end{proof}



\end{document}